\documentclass[10pt,reqno]{amsart}
\usepackage[hypertex]{hyperref}

\usepackage{amsmath,amsthm}
\usepackage{amssymb}
\usepackage{euscript}
\usepackage[all,tips]{xy}

\numberwithin{equation}{subsection}

\newcommand{\sqsp}{\renewcommand{\baselinestretch}{1.15}\tiny\normalsize}

\newcommand{\relphantom[1]}{\mathrel{\phantom{#1}}}

\newtheorem{theorem}[subsection]{Theorem}
\newtheorem{lemma}[subsection]{Lemma}
\newtheorem{proposition}[subsection]{Proposition}
\newtheorem{corollary}[subsection]{Corollary}

\theoremstyle{definition}
\newtheorem{definition}[subsection]{Definition}
\newtheorem{example}[subsection]{Example}
\newtheorem{remark}[subsection]{Remark}

\newcommand{\bk}{\mathbf{k}}
\newcommand{\bQ}{\mathbb{Q}}
\newcommand{\Qbar}{\overline{\mathbb{Q}}}

\newcommand{\muop}{\mu^{op}}
\newcommand{\mustarop}{\mu^{*op}}
\newcommand{\mun}{\mu^{(n)}}
\newcommand{\muone}{\mu^{(1)}}
\newcommand{\mualpha}{\mu_\alpha}

\newcommand{\Deltaop}{\Delta^{op}}
\newcommand{\Deltan}{\Delta^{(n)}}
\newcommand{\Deltaone}{\Delta^{(1)}}
\newcommand{\Deltaoneop}{\Delta^{(1)op}}
\newcommand{\Deltaalpha}{\Delta_\alpha}
\newcommand{\Deltaalphaop}{\Delta_\alpha^{op}}

\DeclareMathOperator{\Hom}{Hom}

\DeclareMathOperator{\ad}{ad}
\DeclareMathOperator{\Gal}{Gal}

\newcommand{\nicearrow}{\SelectTips{cm}{10}}

\newcommand{\vzero}{{\begin{bmatrix}0 & 0 \\ 0 & 1\end{bmatrix}}}
\newcommand{\vone}{{\begin{bmatrix}1 & 0 \\ 0 & 0\end{bmatrix}}}
\newcommand{\phione}{{\begin{bmatrix}0 & 0 \\ (1,0) & 0\end{bmatrix}}}
\newcommand{\phitwo}{{\begin{bmatrix}0 & 0 \\ (0,1) & 0\end{bmatrix}}}
\newcommand{\krmatrix}{{\begin{bmatrix}a & 0 \\ (c,d) & b\end{bmatrix}}}
\newcommand{\kralphat}{{\begin{bmatrix}a & 0 \\ (c+d,0) & b\end{bmatrix}}}
\newcommand{\kralphab}{{\begin{bmatrix}a & 0 \\ (0,c+d) & b\end{bmatrix}}}
\newcommand{\kralphap}{{\begin{bmatrix}a & 0 \\ (d,c) & b\end{bmatrix}}}

\newcommand{\eoneone}{{\begin{bmatrix}1 & 0 \\ 0 & 0\end{bmatrix}}}
\newcommand{\eonetwo}{{\begin{bmatrix}0 & 1 \\ 0 & 0\end{bmatrix}}}

\newcommand{\triangular}{{\begin{bmatrix}\bk & 0 & 0\\ \bk & \bk & 0\\ \bk & 0 & \bk\end{bmatrix}}}
\newcommand{\qmatrix}{{\begin{bmatrix}a & 0 & 0\\ d & b & 0\\ e & 0 & c\end{bmatrix}}}
\newcommand{\qmatrixalpha}{{\begin{bmatrix}a & 0 & 0\\ e & c & 0\\ d & 0 & b\end{bmatrix}}}
\newcommand{\qmatrixxy}{{\begin{bmatrix}0 & 0 & 0\\ 0 & d & 0\\ 0 & 0 & e\end{bmatrix}}}
\newcommand{\qmatrixyx}{{\begin{bmatrix}0 & 0 & 0\\ 0 & e & 0\\ 0 & 0 & d\end{bmatrix}}}
\newcommand{\eone}{{\begin{bmatrix}1 & 0 & 0 \\ 0 & 0 & 0\\ 0 & 0 & 0\end{bmatrix}}}

\begin{document}

\title{Infinitesimal Hom-bialgebras and Hom-Lie bialgebras}
\author{Donald Yau}

\begin{abstract}
We study the Hom-type generalization of infinitesimal bialgebras, called infinitesimal Hom-bialgebras.  In particular, we consider infinitesimal Hom-bialgebras arising from quivers, the sub-classes of coboundary and quasi-triangular infinitesimal Hom-bialgebras, the associative Hom-Yang-Baxter equation, and homological perturbation of the comultiplications in infinitesimal Hom-bialgebras.  The relationships between infinitesimal Hom-bialgebras, Hom-Lie bialgebras, and the classical Hom-Yang-Baxter equation are also studied.
\end{abstract}



\address{Department of Mathematics\\
    The Ohio State University at Newark\\
    1179 University Drive\\
    Newark, OH 43055, USA}
\email{dyau@math.ohio-state.edu}

\date{\today}
\maketitle

\sqsp

\section{Introduction}


An infinitesimal bialgebra $(A,\mu,\Delta)$, often abbreviated to $\epsilon$-bialgebra, is simultaneously an associative algebra $(A,\mu)$ and a coassociative coalgebra $(A,\Delta)$, in which the comultiplication $\Delta$ is a $1$-cocycle in Hochschild cohomology (i.e., a derivation) with coefficients in $A^{\otimes 2}$.  Infinitesimal bialgebras were introduced by Joni and Rota \cite{jr}, called infinitesimal coalgebras there, in the context of the calculus of divided differences.  In combinatorics, they are further studied in \cite{aguiar3,er,foissy,hr}, among others.


Many properties and examples of $\epsilon$-bialgebras were established by Aguiar \cite{aguiar,aguiar2,aguiar4}.  In particular, coboundary and quasi-triangular $\epsilon$-bialgebras, infinitesimal Hopf algebras, and the notion of a Drinfel'd double for a finite dimensional $\epsilon$-bialgebra were developed in those papers.  Moreover, it was established in \cite{aguiar} that the path algebra of an arbitrary quiver has the richer structure of an $\epsilon$-bialgebra.  The associative Yang-Baxter equation (AYBE), which is closely related to coboundary and quasi-triangular $\epsilon$-bialgebras, was introduced in \cite{aguiar} and was further studied in \cite{aguiar2,pol,schedler}.

Infinitesimal bialgebras are closely related to Drinfel'd's Lie bialgebras \cite{dri83,dri87}.  A Lie bialgebra is simultaneously a Lie algebra and a Lie coalgebra, in which the cobracket is a $1$-cocycle in Chevalley-Eilenberg cohomology.  Thus, the cocycle condition in an $\epsilon$-bialgebra can be seen as an associative analog of the cocycle condition in a Lie bialgebra.  In fact, in \cite{aguiar2} necessary and sufficient conditions were found so that an $\epsilon$-bialgebra gives rise to a Lie bialgebra via the (co)commutator (co)bracket.  Moreover, it was shown in \cite{aguiar2} that, under suitable conditions, solutions of the AYBE give rise to solutions of the classical Yang-Baxter equation (CYBE) \cite{skl1,skl2}.


The purpose of this paper is to study the Hom-type generalization of $\epsilon$-bialgebras, called infinitesimal Hom-bialgebras, often abbreviated to $\epsilon$-Hom-bialgebras.  Hom-type algebras are usually defined by twisting the defining axioms of a type of algebras by a certain twisting map.  When the twisting map happens to be the identity map, one gets an ordinary algebraic structure.  Hom-Lie algebras were introduced in \cite{hls} to describe the structures on certain deformations of the Witt and the Virasoro algebras.  Corresponding to Hom-Lie algebras are Hom-associative algebras \cite{ms}, which give rise to Hom-Lie algebras via the commutator bracket.  Conversely, there is an enveloping Hom-bialgebra associated to each Hom-Lie algebra \cite{yau,yau3}.  Further recent studies of Hom-type structures include \cite{ama,arms,ams,cg}, \cite{fg1} - \cite{gohr}, \cite{mak}, \cite{ms2} - \cite{ms4}, and \cite{yau} - \cite{yau11}, among others.


There are two conceptual ways to think about $\epsilon$-Hom-bialgebras.  On the one hand, an $\epsilon$-Hom-bialgebra is a non-(co)associative analog of an $\epsilon$-bialgebra, in which the non-(co)associativity is controlled by a twisting map $\alpha$.  The cocycle condition in an $\epsilon$-bialgebra is also replaced by a suitable cocycle condition in Hom-algebra cohomology.  On the other hand, as we will show below, $\epsilon$-Hom-bialgebras are the Hom-associative analogs of the author's Hom-Lie bialgebras \cite{yau8}.  Moreover, we will define a Hom-type analog of the AYBE, called the associative Hom-Yang-Baxter equation (AHYBE), which is related to the classical Hom-Yang-Baxter equation (CHYBE) \cite{yau8} as the AYBE is related to the CYBE.

A description of the rest of this paper follows.

In section \ref{sec:infinitesimal} we define $\epsilon$-Hom-bialgebras and prove some of their basic properties, including two general construction results called the Twisting Principles.  Briefly, an $\epsilon$-Hom-bialgebra $(A,\mu,\Delta,\alpha)$ consists of a Hom-associative algebra $(A,\mu,\alpha)$ and a Hom-coassociative coalgebra $(A,\Delta,\alpha)$ such that $\Delta$ is a $1$-cocycle in Hom-algebra cohomology.  An $\epsilon$-Hom-bialgebra with $\alpha = Id$ is exactly an $\epsilon$-bialgebra.  The First Twisting Principle (Theorem \ref{thm:firsttp}) says that an $\epsilon$-bialgebra $A$ and a morphism $\alpha \colon A \to A$ of $\epsilon$-bialgebras give rise to an $\epsilon$-Hom-bialgebra $A_\alpha$ with twisted (co)multiplication.  This result allows us to construct multiple $\epsilon$-Hom-bialgebras from a given $\epsilon$-bialgebra.  A twisting result of this form first appeared in \cite{yau2} (Theorem 2.3).  That result and its generalizations have been employed by various authors; see, for example, \cite{ama} (Theorem 2.7), \cite{ams} (Theorems 1.7 and 2.6), \cite{fg2} (Section 2), \cite{gohr} (Proposition 1), \cite{mak} (Theorems 2.1 and 3.5), \cite{ms4} (Theorem 3.15 and Proposition 3.30), and \cite{yau3} - \cite{yau11}.

As illustrations of the First Twisting Principle, we observe that every quiver $Q$ and a suitably defined quiver morphism $\alpha \colon Q \to Q$ give rise to an $\epsilon$-Hom-bialgebra structure on the space $\bk Q$ of its path algebra (Corollary \ref{cor:quiver}).  For instance, starting from the Kronecker quiver, we construct three associated Kronecker $\epsilon$-Hom-bialgebras (Example \ref{ex:kronecker}).  The Second Twisting Principle (Theorem \ref{thm:secondtp}) says that every $\epsilon$-Hom-bialgebra gives rise to a derived sequence of $\epsilon$-Hom-bialgebras with suitably twisted structure maps.  In Example \ref{ex:kronecker} we also observe that the Kronecker $\epsilon$-bialgebra ($=$ the path $\epsilon$-bialgebra of the Kronecker quiver) can be recovered from one of the Kronecker $\epsilon$-Hom-bialgebras using the Second Twisting Principle.

In section \ref{sec:cobound} we study the sub-class of coboundary $\epsilon$-Hom-bialgebras.  They are the Hom-type analogs of Aguiar's coboundary $\epsilon$-bialgebras \cite{aguiar,aguiar2}.  A coboundary $\epsilon$-Hom-bialgebra is an $\epsilon$-Hom-bialgebra in which the comultiplication is a principal derivation $[-,r]_*$ (in the Hom-algebra sense) for some $\alpha$-invariant $2$-tensor $r$.  We prove the two Twisting Principles for coboundary $\epsilon$-Hom-bialgebras (Theorems \ref{thm:coboundtp} and \ref{thm:coboundtp2}).  In Theorem \ref{thm:coboundchar}, given a Hom-associative algebra $(A,\mu,\alpha)$ and an $\alpha$-invariant $2$-tensor $r$, we show that the comultiplication $\Delta = [-,r]_*$ gives a coboundary $\epsilon$-Hom-bialgebra $(A,\mu,\Delta,\alpha)$ if and only if a certain element $A(r) \in A^{\otimes 3}$ is $A$-invariant in a suitable sense.

In section \ref{sec:pertubation} we study perturbation of the comultiplications in $\epsilon$-Hom-bialgebras.   Given an $\epsilon$-Hom-bialgebra $(A,\mu,\Delta,\alpha)$ and an $\alpha$-invariant $2$-tensor $r$, we consider the perturbed comultiplication $\Delta' = \Delta + [-,r]_*$.  From the point-of-view of homological algebra, $\Delta'$ is the perturbation of the cocycle $\Delta$ by the coboundary $[-,r]_*$.  In Theorem \ref{thm:perturb} we give a necessary and sufficient condition under which $(A,\mu,\Delta',\alpha)$ becomes an $\epsilon$-Hom-bialgebra.  This condition also involves the element $A(r) \in A^{\otimes 3}$.  Our perturbation theory of $\epsilon$-Hom-bialgebras is similar to Drinfel'd's perturbation theory of quasi-Hopf algebras  \cite{dri83b,dri89b,dri90,dri91,dri92}.  This perturbation question does not seem to have been studied before even for $\epsilon$-bialgebras.  In particular, if we restrict to the special case $\alpha = Id$ in Theorem \ref{thm:perturb}, we obtain a necessary and sufficient condition under which an $\epsilon$-bialgebra gives rise to another $\epsilon$-bialgebra with perturbed comultiplication.

In section \ref{sec:ahybe} we study the associative Hom-Yang-Baxter equation (AHYBE) and quasi-triangular $\epsilon$-Hom-bialgebras.  The former is defined as the equation $A(r) = 0$ (in a Hom-associative algebra $(A,\mu,\alpha)$), which reduces to Aguiar's AYBE when $\alpha = Id$.  We prove the two Twisting Principles for solutions of the AHYBE (Theorems \ref{thm:ahybe1} and \ref{thm:ahybe2}).  A quasi-triangular $\epsilon$-Hom-bialgebra is defined as a coboundary $\epsilon$-Hom-bialgebra in which $r$ is a solution of the AHYBE.  Again, this is the Hom-type analog of Aguiar's quasi-triangular $\epsilon$-bialgebras \cite{aguiar,aguiar2}.  There are Twisting Principles for quasi-triangular $\epsilon$-Hom-bialgebras (Corollaries \ref{cor:qt1} and \ref{cor:qt2}).  In Theorem \ref{thm:cobahybe} we give several equivalent characterizations of the AHYBE in a coboundary $\epsilon$-Hom-bialgebra.

In section \ref{sec:homlie} we study how $\epsilon$-Hom-bialgebras give rise to Hom-Lie bialgebras, generalizing some of the results in \cite{aguiar2}.  A Hom-Lie bialgebra is simultaneously a Hom-Lie algebra and a Hom-Lie coalgebra, in which the cobracket is a $1$-cocycle in Hom-Lie algebra cohomology.  Hom-Lie bialgebras were first introduced and studied by the author in \cite{yau8} as the Hom-type analogs of Drinfel'd's Lie bialgebras \cite{dri83,dri87}.  As in the case of associative algebras, given a Hom-associative algebra, its commutator algebra is a Hom-Lie algebra \cite{ms}.  The dual result for Hom-coassociative coalgebra and Hom-Lie coalgebra also holds.  Therefore, it is natural to ask whether an $\epsilon$-Hom-bialgebra gives rise to a Hom-Lie bialgebra with its (co)commutator (co)bracket.  The answer to this question turns out to depend on a certain map on the tensor-square.

Given any $\epsilon$-Hom-bialgebra $A$, there is a map $B \colon A^{\otimes 2} \to A^{\otimes 2}$ called the \emph{balanceator}, which is the Hom-type analog of a map of the same name for an $\epsilon$-bialgebra \cite{aguiar2}.  In Corollary \ref{cor:homliebi} we establish that an $\epsilon$-Hom-bialgebra $A$ gives rise to a Hom-Lie bialgebra with its commutator bracket and cocommutator cobracket precisely when its balanceator is symmetric.  Then we prove the two Twisting Principles for $\epsilon$-Hom-bialgebras with symmetric balanceators (Theorems \ref{thm:Btp} and \ref{thm:Btp2}).  If the $\epsilon$-Hom-bialgebra $A$ in question is coboundary with $\Delta = [-,r]_*$ and $r$ anti-symmetric, then its balanceator is trivial and hence symmetric (Lemma \ref{lem1:coboundhomlie}).  In this case, we obtain a coboundary Hom-Lie bialgebra using the (co)commutator (co)bracket (Theorem \ref{thm:coboundhomlie}).

In section \ref{sec:qt} we study how solutions of the AHYBE and quasi-triangular $\epsilon$-Hom-bialgebras give rise to solutions of the classical Hom-Yang-Baxter equation (CHYBE) and quasi-triangular Hom-Lie bialgebras, respectively.  The CHYBE was introduced by the author in \cite{yau8} as the Hom-type analog of the CYBE.  In Theorem \ref{thm:chybe} we show that a solution $r$ of the AHYBE in a Hom-associative algebra is also a solution of the CHYBE in the commutator Hom-Lie algebra, provided that $r$ is either symmetric or anti-symmetric.  Using this result and Theorem \ref{thm:coboundhomlie}, we infer that a quasi-triangular $\epsilon$-Hom-bialgebra with $r$ anti-symmetric gives rise to a quasi-triangular Hom-Lie bialgebra with its (co)commutator (co)bracket (Corollary \ref{cor:qthomlie}).

\section{Infinitesimal Hom-bialgebras}
\label{sec:infinitesimal}

In this section, we define infinitesimal Hom-bialgebras, prove some construction results, and consider some concrete examples.  Infinitesimal Hom-bialgebras arising from quivers are considered near the end of this section.

\subsection{Conventions}
Throughout the rest of this paper, we work over a fixed field $\bk$ of characteristic $0$.  Vector spaces, tensor products, linearity, and $\Hom$ are all meant over $\bk$. If $f \colon V \to V$ is a linear self-map on a vector space $V$, then $f^n \colon V \to V$ denotes the composition $f \circ \cdots \circ f$ of $n$ copies of $f$, with $f^0 = Id$.  For a map $\Delta \colon V \to V^{\otimes 2}$, we use Sweedler's notation \cite{sweedler} $\Delta(a) = \sum_{(a)} a_1 \otimes a_2$.  For an element $r = \sum_i u_i \otimes v_i \in V^{\otimes 2}$, the summation sign will often be omitted in computations to simplify the typography.  For a map $\mu \colon V^{\otimes 2} \to V$, we often write $\mu(a,b)$ as $ab$ for $a,b \in V$.  If $W$ is another vector space, then $\tau \colon V \otimes W \cong W \otimes V$ denotes the twist isomorphism, $\tau(v \otimes w) = w \otimes v$.


\begin{definition}
\label{def:homas}
\begin{enumerate}
\item
A \textbf{Hom-associative algebra} \cite{ms} $(A,\mu,\alpha)$ consists of a  $\bk$-module $A$, a bilinear map $\mu \colon A^{\otimes 2} \to A$ (the multiplication), and a linear self-map $\alpha \colon A \to A$ (the twisting map) such that:
\begin{equation}
\label{homassaxioms}
\begin{split}
\alpha \circ \mu &= \mu \circ \alpha^{\otimes 2} \quad \text{(multiplicativity)},\\
\mu \circ (\alpha \otimes \mu) &= \mu \circ (\mu \otimes \alpha) \quad \text{(Hom-associativity)}.
\end{split}
\end{equation}
A morphism of Hom-associative algebras is a linear map of the underlying $\bk$-modules that commutes with the twisting maps and the multiplications.
\item
A \textbf{Hom-coassociative coalgebra} \cite{ms2,ms4} $(C,\Delta,\alpha)$ consists of a $\bk$-module $C$, a linear map $\Delta \colon C \to C^{\otimes 2}$ (the comultiplication), and a linear self-map $\alpha \colon C \to C$ (the twisting map) such that:
\begin{equation}
\label{homcoassaxioms}
\begin{split}
\alpha^{\otimes 2} \circ \Delta &= \Delta \circ \alpha \quad \text{(comultiplicativity)},\\
(\alpha \otimes \Delta) \circ \Delta &= (\Delta \otimes \alpha) \circ \Delta \quad \text{(Hom-coassociativity)}.
\end{split}
\end{equation}
A morphism of Hom-coassociative coalgebras is a linear map of the underlying $\bk$-modules that commutes with the twisting maps and the comultiplications.
\item
An \textbf{infinitesimal Hom-bialgebra} (often abbreviated to \textbf{$\epsilon$-Hom-bialgebra}) is a quadruple $(A,\mu,\Delta,\alpha)$, in which $(A,\mu,\alpha)$ is a Hom-associative algebra, $(A,\Delta,\alpha)$ is a Hom-coassociative coalgebra, and the condition
\begin{equation}
\label{cocycle}
\Delta \circ \mu = (\mu \otimes \alpha) \circ (\alpha \otimes \Delta) + (\alpha \otimes \mu) \circ (\Delta \otimes \alpha)
\end{equation}
holds.  An \textbf{infinitesimal bialgebra} (often abbreviated to \textbf{$\epsilon$-bialgebra}) is an $\epsilon$-Hom-bialgebra with $\alpha = Id$.  A morphism of $\epsilon$-Hom-bialgebras is a linear map that commutes with the twisting maps, the multiplications, and the comultiplications.
\end{enumerate}
\end{definition}

In terms of elements and writing $\mu(a,b)$ as $ab$, the condition \eqref{cocycle} can be rewritten as
\begin{equation}
\label{cocycle'}
\Delta(ab) = \sum_{(b)} \alpha(a)b_1 \otimes \alpha(b_2) + \sum_{(a)} \alpha(a_1) \otimes a_2\alpha(b)
\end{equation}
for all $a,b \in A$.   We will sometimes denote an $\epsilon$-Hom-bialgebra $(A,\mu,\Delta,\alpha)$ simply by $A$.

To be more explicit, an $\epsilon$-bialgebra is a tuple $(A,\mu,\Delta)$ consisting of an associative algebra $(A,\mu)$ and a coassociative coalgebra $(A,\Delta)$ such that the condition
\begin{equation}
\label{cocycleinf}
\Delta \circ \mu = (\mu \otimes Id) \circ (Id \otimes \Delta) + (Id \otimes \mu) \circ (\Delta \otimes Id)
\end{equation}
holds.  The condition \eqref{cocycleinf} is equivalent to
\[
\Delta(ab) = \sum_{(b)} ab_1 \otimes b_2 + \sum_{(a)} a_1 \otimes a_2b
\]
for all $a,b \in A$.  In other words, \eqref{cocycleinf} says that $\Delta \colon A \to A^{\otimes 2}$ is a derivation of $A$ with values in the $A$-bimodule $A^{\otimes 2}$, in which the left and the right $A$-actions are
\begin{equation}
\label{a2bimodule}
a \cdot (b_1 \otimes b_2) = ab_1 \otimes b_2 \quad\text{and}\quad
(b_1 \otimes b_2) \cdot a = b_1 \otimes b_2a.
\end{equation}
A morphism of $\epsilon$-bialgebras is a map that is simultaneously a morphism of associative algebras and of coassociative coalgebras.
Infinitesimal bialgebras were introduced by Joni and Rota \cite{jr}. Our terminology follows that of Aguiar \cite{aguiar}.

\begin{remark}
\label{rk:cocycle}
The condition \eqref{cocycle} is, in fact, a cocycle condition in Hom-associative algebra cohomology \cite{ms3} with non-trivial coefficients.  Indeed, for a Hom-associative algebra $(A,\mu,\alpha)$, one can regard $A^{\otimes n}$ as an $A$-bimodule (in the Hom-associative sense) with left and right $A$-actions:
\begin{equation}
\label{dotaction}
\begin{split}
a \bullet (b_1 \otimes \cdots \otimes b_n) &= \mu(\alpha(a),b_1) \otimes \alpha(b_2) \otimes \cdots \otimes \alpha(b_n),\\
(b_1 \otimes \cdots \otimes b_n) \bullet a &= \alpha(b_1) \otimes \cdots \otimes \alpha(b_{n-1}) \otimes \mu(b_n,\alpha(a)).
\end{split}
\end{equation}
With these notations, the condition \eqref{cocycle} becomes
\begin{equation}
\label{cocycle''}
\Delta(ab) = a \bullet \Delta(b) + \Delta(a) \bullet b.
\end{equation}
Set $C^n(A,A^{\otimes 2}) = \Hom(A^{\otimes n},A^{\otimes 2})$.  Now if $A$ is an $\epsilon$-Hom-bialgebra with comultiplication $\Delta \colon A \to A^{\otimes 2}$, then we can regard $\Delta \in C^1(A,A^{\otimes 2})$.  Generalizing Definition 3.4 in \cite{ms3}, the $1$-coboundary operator $\delta^1 \colon C^1(A,A^{\otimes 2}) \to C^2(A,A^{\otimes 2})$ is given by
\begin{equation}
\label{delta1}
(\delta^1 f)(a,b) = f(ab) - f(a)\bullet b - a \bullet f(b).
\end{equation}
In particular, the condition \eqref{cocycle} (in the form \eqref{cocycle''}) is equivalent to $\delta^1\Delta = 0$, i.e., $\Delta$ is a $1$-cocycle in $C^1(A,A^{\otimes 2})$.
\end{remark}

\begin{example}
\label{ex:trivialmult}
A Hom-associative algebra $(A,\mu,\alpha)$ becomes an $\epsilon$-Hom-bialgebra when equipped with the trivial comultiplication $\Delta = 0$.  Likewise, a Hom-coassociative coalgebra $(C,\Delta,\alpha)$ becomes an $\epsilon$-Hom-bialgebra when equipped with the trivial multiplication $\mu = 0$.\qed
\end{example}

In the following three results, we give some closure properties of the class of $\epsilon$-Hom-bialgebras.  The first two results are the Hom analogs of the corresponding observations in (\cite{aguiar} page 3).

\begin{proposition}
\label{prop:plusminus}
Let $(A,\mu,\Delta,\alpha)$ be an $\epsilon$-Hom-bialgebra.  Then so are $(A,-\mu,\Delta,\alpha)$, $(A,\mu,-\Delta,\alpha)$, and $(A,\muop,\Deltaop,\alpha)$, where $\muop = \mu \circ \tau$ and $\Deltaop = \tau \circ \Delta$.
\end{proposition}

\begin{proof}
Both $(A,-\mu,\Delta,\alpha)$ and $(A,\mu,-\Delta,\alpha)$ are $\epsilon$-Hom-bialgebras because the axioms \eqref{homassaxioms}, \eqref{homcoassaxioms}, and \eqref{cocycle} clearly still hold when $\mu$ is replaced by $-\mu$ or when $\Delta$ is replaced by $-\Delta$.

For $(A,\muop,\Deltaop,\alpha)$, it is clear that $\alpha$ is multiplicative with respect to $\muop$ and comultiplicative with respect to $\Deltaop$.  Let $\pi \colon A^{\otimes 3} \to A^{\otimes 3}$ be the permutation linear isomorphism given by $\pi(a \otimes b \otimes c) = c \otimes b \otimes a$. The Hom-associativity axiom for $\muop$ now follows from that of $\mu$ and the identities
\[
\begin{split}
\muop \circ (\alpha \otimes \muop) &= \mu \circ (\mu \otimes \alpha) \circ \pi,\\
\muop \circ (\muop \otimes \alpha) &= \mu \circ (\alpha \otimes \mu) \circ \pi.
\end{split}
\]
Likewise the Hom-coassociativity axiom for $\Deltaop$ follows from that of $\Delta$ and the identities
\[
\begin{split}
(\alpha \otimes \Deltaop) \circ \Deltaop &= \pi \circ (\Delta \otimes \alpha) \circ \Delta,\\
(\Deltaop \otimes \alpha) \circ \Deltaop &= \pi \circ (\alpha \otimes \Delta) \circ \Delta.
\end{split}
\]
Finally, the cocycle condition \eqref{cocycle} for $\Deltaop$ and $\muop$ follows from the following calculation:
\[
\begin{split}
\Deltaop \circ \muop &= \tau \circ \Delta \circ \mu \circ \tau\\
&= \tau \circ (\mu \otimes \alpha) \circ (\alpha \otimes \Delta) \circ \tau + \tau \circ (\alpha \otimes \mu) \circ (\Delta \otimes \alpha) \circ \tau\\
&= (\alpha \otimes \muop) \circ (\Deltaop \otimes \alpha) + (\muop \otimes \alpha) \circ (\alpha \otimes \Deltaop).
\end{split}
\]
We have shown that $(A,\muop,\Deltaop,\alpha)$ is an $\epsilon$-Hom-bialgebra.
\end{proof}

For a $\bk$-module $V$, let $V^*$ denote its linear dual $\Hom(V,\bk)$.  When $V$ is finite dimensional, there is a canonical linear isomorphism $(V^*)^{\otimes 2} \cong (V^{\otimes 2})^*$.  For $\phi \in V^*$ and $v \in V$, we often write $\phi(v)$ as $\langle\phi,v\rangle$.

\begin{proposition}
\label{prop:dual}
Let $(A,\mu,\Delta,\alpha)$ be a finite dimensional $\epsilon$-Hom-bialgebra.  Then so is $(A^*,\Delta^*,\mu^*,\alpha^*)$, where
\[
\langle\alpha^*(\phi),a\rangle = \langle\phi, \alpha(a)\rangle,\quad
\langle\Delta^*(\phi \otimes \psi),a\rangle = \langle\phi \otimes \psi,\Delta(a)\rangle, \quad
\langle\mu^*(\phi), a\otimes b\rangle = \langle\phi,\mu(a \otimes b)\rangle
\]
for $a,b \in A$ and $\phi,\psi \in A^*$.
\end{proposition}

\begin{proof}
One can check directly that $(A^*,\Delta^*,\alpha^*)$ is a Hom-associative algebra (which is true even if $A$ is not finite dimensional) and that $(A^*,\mu^*,\alpha^*)$ is a Hom-coassociative coalgebra (which requires that $A$ be finite dimensional) \cite{ms4} (Corollary 4.12).  It remains to establish the cocycle condition \eqref{cocycle} for $\Delta^*$ and $\mu^*$.  For $a,b \in A$ and $\phi,\psi \in A^*$, we compute as follows:
\[
\begin{split}
\langle (\mu^* \circ \Delta^*)(\phi \otimes \psi), a \otimes b\rangle &= \langle \phi \otimes \psi, (\Delta \circ \mu)(a \otimes b)\rangle\\
&= \langle \phi \otimes \psi, (\mu \otimes \alpha) \circ (\alpha \otimes \Delta)(a \otimes b)\rangle\\
&\relphantom{} + \langle \phi \otimes \psi, (\alpha \otimes \mu) \circ (\Delta \otimes \alpha)(a \otimes b)\rangle\\
&= \langle (\alpha^* \otimes \Delta^*) \circ (\mu^* \otimes \alpha^*)(\phi \otimes \psi), (a \otimes b)\rangle\\
&\relphantom{} + \langle(\Delta^* \otimes \alpha^*) \circ (\alpha^* \otimes \mu^*)(\phi \otimes \psi), (a \otimes b)\rangle.
\end{split}
\]
We have shown that the cocycle condition \eqref{cocycle} holds for $\Delta^*$ and $\mu^*$.
\end{proof}

The following result shows that every $\epsilon$-Hom-bialgebra gives rise to a derived sequence of $\epsilon$-Hom-bialgebras with twisted (co)multiplications and twisting maps.  Such a twisting result for Hom-type algebras is referred to as the Second Twisting Principle in \cite{yau11}.  We will use it again in later sections to prove the Second Twisting Principle for coboundary $\epsilon$-Hom-bialgebras.

\begin{theorem}
\label{thm:secondtp}
Let $(A,\mu,\Delta,\alpha)$ be an $\epsilon$-Hom-bialgebra.  Then so is
\[
A^n = (A,\mu^{(n)},\Delta^{(n)},\alpha^{2^n})
\]
for each $n \geq 0$, where $\mu^{(n)} = \alpha^{2^n-1}\circ\mu$ and $\Delta^{(n)} = \Delta \circ \alpha^{2^n-1}$.
\end{theorem}

\begin{proof}
First note that $A^0 = A$, $A^1 = (A,\muone = \alpha\circ\mu, \Deltaone = \Delta \circ \alpha, \alpha^2)$, and $A^{n+1} = (A^n)^1$.  Therefore, by an induction argument, it suffices to prove the case $n=1$, i.e., that $A^1$ is an $\epsilon$-Hom-bialgebra.  One can check directly that $(A,\muone,\alpha^2)$ is a Hom-associative algebra and that $(A,\Deltaone,\alpha^2)$ is a Hom-coassociative coalgebra, as was done in \cite{yau11}.  It remains to establish the cocycle condition \eqref{cocycle} in $A^1$.  Using $\muone = \alpha \circ \mu = \mu \circ \alpha^{\otimes 2}$, $\Deltaone = \Delta \circ \alpha = \alpha^{\otimes 2} \circ \Delta$, and the cocycle condition \eqref{cocycle} in $A$, we compute as follows:
\[
\begin{split}
\Deltaone \circ \muone &= (\alpha^2)^{\otimes 2} \circ \Delta \circ \mu\\
&= ((\alpha^2 \circ \mu) \otimes \alpha^3) \circ (\alpha \otimes \Delta) + (\alpha^3 \otimes (\alpha^2 \circ \mu)) \circ (\Delta \otimes \alpha)\\
&= (\muone \otimes \alpha^2) \circ \alpha^{\otimes 3} \circ (\alpha \otimes \Delta) + (\alpha^2 \otimes \muone) \circ \alpha^{\otimes 3} \circ (\Delta \otimes \alpha)\\
&= (\muone \otimes \alpha^2) \circ (\alpha^2 \otimes \Deltaone) + (\alpha^2 \otimes \muone) \circ (\Deltaone \otimes \alpha^2).
\end{split}
\]
We have shown that $A^1$ is an $\epsilon$-Hom-bialgebra, as desired.
\end{proof}

The $\epsilon$-Hom-bialgebra $A^n$ in Theorem \ref{thm:secondtp} is referred to as the \textbf{$n$th derived $\epsilon$-Hom-bialgebra} of $A$.

In the following result, it is shown that every $\epsilon$-bialgebra can be twisted into an $\epsilon$-Hom-bialgebra via any self-morphism.  It is the main tool with which we construct concrete examples of $\epsilon$-Hom-bialgebras.  Such a twisting result for Hom-type algebras is referred to as the First Twisting Principle in \cite{yau11}.  The First Twisting Principle was first established for $G$-Hom-associative algebras (which include Hom-associative and Hom-Lie algebras as special cases) in \cite{yau2} (Theorem 2.3).  Its variations and generalizations have appeared in many papers on Hom-type algebras, including \cite{ama} (Theorem 2.7), \cite{ams} (Theorems 1.7 and 2.6), \cite{fg2} (Section 2), \cite{gohr} (Proposition 1), \cite{mak} (Theorems 2.1 and 3.5), and \cite{ms4} (Theorem 3.15 and Proposition 3.30).  It is the major method with which concrete examples of Hom-type algebras are constructed.

\begin{theorem}
\label{thm:firsttp}
Let $(A,\mu,\Delta)$ be an $\epsilon$-bialgebra and $\alpha \colon A \to A$ be a morphism of $\epsilon$-bialgebras.  Then
\[
A_\alpha = (A,\mu_\alpha,\Delta_\alpha,\alpha)
\]
is an $\epsilon$-Hom-bialgebra, where $\mu_\alpha = \alpha \circ \mu$ and $\Delta_\alpha = \Delta \circ \alpha$.
\end{theorem}

\begin{proof}
It is already shown in \cite{yau2} (Theorem 2.3) that $(A,\mu_\alpha,\alpha)$ is a Hom-associative algebra.  The dual argument \cite{ms4} shows that $(A,\Delta_\alpha,\alpha)$ is a Hom-coassociative coalgebra.  It remains to establish the cocycle condition \eqref{cocycle} in $A_\alpha$.  Using $\mualpha = \alpha \circ \mu = \mu \circ \alpha^{\otimes 2}$, $\Deltaalpha = \Delta \circ \alpha = \alpha^{\otimes 2} \circ \Delta$, and the cocycle condition \eqref{cocycleinf} in the $\epsilon$-bialgebra $A$, we compute as follows:
\[
\begin{split}
\Deltaalpha \circ \mualpha &= (\alpha^2)^{\otimes 2} \circ \Delta \circ \mu\\
&= (\alpha^2)^{\otimes 2} \circ (\mu \otimes Id) \circ (Id \otimes \Delta) +  (\alpha^2)^{\otimes 2} \circ(Id \otimes \mu) \circ (\Delta \otimes Id)\\
&= ((\mualpha \circ \alpha^{\otimes 2}) \otimes \alpha^2) \circ (Id \otimes \Delta) + (\alpha^2 \otimes (\mualpha \circ \alpha^{\otimes 2})) \circ (\Delta \otimes Id)\\
&= (\mualpha \otimes \alpha) \circ \alpha^{\otimes 3} \circ (Id \otimes \Delta) + (\alpha \otimes \mualpha) \circ \alpha^{\otimes 3} \circ (\Delta \otimes Id)\\
&= (\mualpha \otimes \alpha) \circ (\alpha \otimes \Deltaalpha) + (\alpha \otimes \mualpha) \circ (\Deltaalpha \otimes \alpha).
\end{split}
\]
We have shown that $A_\alpha$ is an $\epsilon$-Hom-bialgebra.
\end{proof}

In order to use Theorem \ref{thm:firsttp} to construct $\epsilon$-Hom-bialgebras, we need to be able to construct $\epsilon$-bialgebra morphisms.  The following result is useful for this purpose.

\begin{proposition}
\label{prop:morphism}
Let $(A,\mu,\Delta)$ be an $\epsilon$-bialgebra and $\alpha \colon A \to A$ be an associative algebra morphism.  Suppose that $X$ is a set of associative algebra generators for $A$ and that $\Delta \circ \alpha$ and $\alpha^{\otimes 2} \circ \Delta$ are equal when restricted to $X$.  Then $\alpha$ is an $\epsilon$-bialgebra morphism.
\end{proposition}

\begin{proof}
We need to show that $\Delta \circ \alpha = \alpha^{\otimes 2} \circ \Delta$ on all of $A$. Since $X$ generates $A$ as an associative algebra, it suffices to show that, if $\Delta(\alpha(a)) = \alpha^{\otimes 2}(\Delta(a))$ and $\Delta(\alpha(b)) = \alpha^{\otimes 2}(\Delta(b))$, then $\Delta(\alpha(ab)) = \alpha^{\otimes 2}(\Delta(ab))$.  We compute as follows:
\[
\begin{split}
\Delta(\alpha(ab)) &= \Delta(\alpha(a)\alpha(b))\\
&= \alpha(a)\alpha(b)_1 \otimes \alpha(b)_2 + \alpha(a)_1 \otimes \alpha(a)_2\alpha(b)\\
&= \alpha(a)\alpha(b_1) \otimes \alpha(b_2) + \alpha(a_1) \otimes \alpha(a_2)\alpha(b)\\
&= \alpha^{\otimes 2}(ab_1 \otimes b_2 + a_1 \otimes a_2b)\\
&= \alpha^{\otimes 2}(\Delta(ab)).
\end{split}
\]
In the second and the last equalities above, we used the cocycle condition \eqref{cocycleinf} in $A$.  In the third equality, we used the assumption that $\Delta \circ \alpha$ and $\alpha^{\otimes 2}\circ \Delta$ are equal on $a$ and $b$.  In the first and the fourth equalities, we used the assumption that $\alpha$ is an associative algebra morphism.
\end{proof}

The rest of this section is devoted to examples of $\epsilon$-Hom-bialgebras that can be constructed using the First Twisting Principle (Theorem \ref{thm:firsttp}).

\begin{example}[\textbf{Dual numbers $\epsilon$-Hom-bialgebra}]
\label{ex:dualnumber}
The algebra of dual numbers $\bk[x]/(x^2)$ is an $\epsilon$-bialgebra with $\Delta(1) = 0$ and $\Delta(x) = x \otimes x$ (\cite{aguiar} Example 2.3.6).  We claim that there are exactly two $\epsilon$-bialgebra morphisms $\alpha \colon \bk[x]/(x^2) \to \bk[x]/(x^2)$ that preserve $1$, namely, $\alpha = Id$ and $\alpha(c + dx) = c$.  Indeed, a unit-preserving algebra morphism on $\bk[x]/(x^2)$ is determined by $\alpha(x) = a + bx$.  Since $\Delta(\alpha(1)) = 0 = \alpha^{\otimes 2}(\Delta(1))$, by Proposition \ref{prop:morphism} $\alpha$ is an $\epsilon$-bialgebra morphism if and only if $\Delta(\alpha(x)) = \alpha^{\otimes 2}(\Delta(x))$.  We have
\[
\Delta(\alpha(x)) = \Delta(a + bx) = bx \otimes x
\]
and
\[
\begin{split}
\alpha^{\otimes 2}(\Delta(x)) &= \alpha(x) \otimes \alpha(x)\\
&=a \otimes a + a \otimes bx + bx \otimes a + bx \otimes bx.
\end{split}
\]
Therefore, $\alpha$ is an $\epsilon$-bialgebra morphism if and only if $a = 0$ and $b \in \{0,1\}$, i.e., $\alpha = Id$ or $\alpha(c + dx) = c$.

Consider the $\epsilon$-bialgebra morphism $\alpha \colon \bk[x]/(x^2) \to \bk[x]/(x^2)$ given by the constant-coefficient projection $\alpha(c + dx) = c$.  By the First Twisting Principle (Theorem \ref{thm:firsttp}), there is a corresponding $\epsilon$-Hom-bialgebra
\[
(\bk[x]/(x^2))_\alpha = (\bk[x]/(x^2),\mualpha = \alpha \circ \mu,\Deltaalpha = \Delta \circ \alpha,\alpha)
\]
with
\[
\mualpha(a+bx,c+dx) = ac \quad\text{and}\quad \Deltaalpha(a+bx) = \Delta(a) = 0
\]
In other words, $(\bk[x]/(x^2))_\alpha$ is the Hom-associative algebra $(\bk[x]/(x^2),\mualpha,\alpha)$ considered as an $\epsilon$-Hom-bialgebra with trivial comultiplication, as in Example \ref{ex:trivialmult}.
\qed
\end{example}

\begin{example}[\textbf{Polynomial $\epsilon$-Hom-bialgebras}]
\label{ex:polynomial}
The polynomial algebra $\bk[x]$ is an $\epsilon$-bialgebra whose comultiplication is determined by $\Delta(1) = 0$ and
\[
\Delta(x^n) = \sum_{i=0}^{n-1} x^i \otimes x^{n-1-i}
\]
for $n \geq 1$ (\cite{aguiar} Example 2.3.3).  This comultiplication is important in the calculus of divided differences \cite{hr,jr}.

We claim that there is a bijection between (the underlying set of) $\bk$ and unit-preserving $\epsilon$-bialgebra morphisms $\alpha_a \colon \bk[x] \to \bk[x]$ determined by $\alpha_a(x) = a + x$ for $a \in \bk$.  Indeed, we have
\[
\Delta(\alpha_a(x)) = \Delta(a+x) = 1 \otimes 1 = \alpha_a^{\otimes 2}(\Delta(x))
\]
and $\Delta(\alpha_a(1)) = 0 = \alpha_a^{\otimes 2}(\Delta(1))$.  Since $\{1,x\}$ generates $\bk[x]$ as an associative algebra, it follows from Proposition \ref{prop:morphism} that each $\alpha_a$ is an $\epsilon$-bialgebra morphism on $\bk[x]$.  Conversely, suppose that $\alpha \colon \bk[x] \to \bk[x]$ is a unit-preserving $\epsilon$-bialgebra morphism.  In particular, $\alpha$ is determined by $\alpha(x) = \sum_{n=0}^r a_nx^n \in \bk[x]$.  On the one hand, we have
\[
\Delta(\alpha(x)) = \alpha^{\otimes 2}(\Delta(x)) = \alpha^{\otimes 2}(1 \otimes 1) = 1 \otimes 1.
\]
On the other hand, we have
\[
\begin{split}
\Delta(\alpha(x)) &= \sum a_n\Delta(x^n)\\
&= a_0\Delta(1) + \sum_{n=1}^r a_n\left(\sum_{i=0}^{n-1} x^i \otimes x^{n-1-i}\right)\\
&= a_1(1 \otimes 1) + a_2(1 \otimes x + x \otimes 1) + \cdots\\
&\relphantom{} + a_r(1 \otimes x^{r-1} + x \otimes x^{r-2} + \cdots + x^{r-1} \otimes 1).
\end{split}
\]
It follows that $a_1 = 1$ and $a_n = 0$ for $n \geq 2$, i.e., $\alpha(x) = a_0 + x = \alpha_{a_0}(x)$.

By the First Twisting Principle (Theorem \ref{thm:firsttp}), for each $a \in \bk$, there is a corresponding $\epsilon$-Hom-bialgebra
\[
\bk[x]_{\alpha_a} = (\bk[x],\mu_{\alpha_a},\Delta_{\alpha_a},\alpha_a).
\]
Its multiplication $\mu_{\alpha_a} = \alpha_a \circ \mu$ is given by
\[
\mu_{\alpha_a}(f,g) = f(a+x)g(a+x)
\]
for $f,g \in \bk[x]$.  Its comultiplication $\Delta_{\alpha_a} = \alpha_a^{\otimes 2} \circ \Delta$ is given by $\Delta_{\alpha_a}(1) = 0$ and
\[
\Delta_{\alpha_a}(x^n) = \sum_{i=0}^{n-1} (a + x)^i \otimes (a+x)^{n-1-i}
\]
for $n \geq 1$.  We can think of the collection $\{\bk[x]_{\alpha_a} \colon a \in \bk\}$ as a $1$-parameter $\epsilon$-Hom-bialgebra deformations of the $\epsilon$-bialgebra $\bk[x]$.  We recover $\bk[x]$ by taking $a=0$.
\qed
\end{example}

An important class of $\epsilon$-bialgebras comes from the path algebras of quivers, as shown in (\cite{aguiar} Example 2.3.2).  In the rest of this section, we discuss how the First Twisting Principle (Theorem \ref{thm:firsttp}) can be applied to these path $\epsilon$-bialgebras to give $\epsilon$-Hom-bialgebras.  Let us briefly recall the relevant notions of quivers and their path algebras.  More detailed discussions of quivers can be found in, e.g., (\cite{ass} Chapter II).

\subsection{Quivers and path $\epsilon$-bialgebras}
\label{subsec:quivers}

A \textbf{quiver} $Q = (Q_0,Q_1,s,t)$ consists of a set $Q_0$ of \textbf{vertices}, a set $Q_1$ of \textbf{arrows}, and two maps $s,t \colon Q_1 \to Q_0$ called the \textbf{source} and the \textbf{target}.  In other words, a quiver is a directed graph.  Graphically, an arrow $\phi \in Q_1$ is represented as a directed edge
\[
s(\phi) \xrightarrow{\phi} t(\phi)
\]
from its source to its target.  A \textbf{path} of length $l \geq 1$ in a quiver $Q$ is an $l$-tuple $\phi_1\cdots\phi_l$ of arrows satisfying $t(\phi_i) = s(\phi_{i+1})$ for $1 \leq i \leq l-1$.  A path of length $0$ is defined as a vertex in $Q_0$.  The set of paths of length $l$ is denoted by $Q_l$.  For a path $\phi_1\cdots\phi_l$, its source and target are defined as $s(\phi_1)$ and $t(\phi_l)$, respectively.

Given a quiver $Q$, its \textbf{path algebra} is the associative algebra $\bk Q = \bigoplus_{l\geq 0} \bk Q_l$ whose underlying $\bk$-module has basis $\cup_{l \geq 0} Q_l$.  The multiplication is determined on the basis elements (i.e., paths) by
\begin{equation}
\label{pathmult}
\mu\left(\phi_1\cdots\phi_l, \psi_1 \cdots \psi_m\right) = \delta_{t(\phi_l),s(\psi_1)}\phi_1\cdots\phi_l\psi_1 \cdots \psi_m.
\end{equation}
In other words, the multiplication on paths is given by the concatenated path if the target of the first path is equal to the source of the second path.  Otherwise, it is $0$.  According to (\cite{aguiar} Example 2.3.2), the path algebra $\bk Q$ of a quiver $Q$ is an $\epsilon$-bialgebra whose comultiplication $\Delta$ is determined by:
\begin{equation}
\label{pathcomult}
\begin{split}
\Delta(v) &= 0 \quad \text{for $v \in Q_0$},\\
\Delta(\phi) &= s(\phi) \otimes t(\phi) \quad \text{for $\phi \in Q_1$, and}\\
\Delta(\phi_1\cdots\phi_l) &= s(\phi_1) \otimes \phi_2 \cdots \phi_l + \phi_1 \cdots \phi_{l-1} \otimes t(\phi_l) + \sum_{i=1}^{l-2} \phi_1\cdots\phi_i \otimes \phi_{i+2}\cdots \phi_l
\end{split}
\end{equation}
for $\phi_1\cdots\phi_l \in Q_l$ with $l \geq 2$.  We refer to this $\epsilon$-bialgebra as the \textbf{path $\epsilon$-bialgebra} of the quiver $Q$.

To use the First Twisting Principle (Theorem \ref{thm:firsttp}) on the path $\epsilon$-bialgebra of a quiver $Q$, we need to be able to construct $\epsilon$-bialgebra morphisms on $\bk Q$.  Such maps can be constructed from suitable morphisms on the quiver $Q$, as we now discuss.

Let $Q$ and $Q'$ be two quivers.  A \textbf{quiver morphism} $\alpha \colon Q \to Q'$ consists of two maps $\alpha_i \colon Q_i \to Q_i'$ $(i = 0,1)$ such that (i) $s(\alpha_1(\phi)) = \alpha_0(s(\phi))$ and $t(\alpha_1(\phi)) = \alpha_0(t(\phi))$ for all $\phi \in Q_1$, and that (ii) $\alpha_0$ is injective.  In other words, a quiver morphism is a map of directed graphs that is injective on vertices and that preserves the sources and the targets of arrows.  In what follows, we will drop the subscripts and write both $\alpha_0$ and $\alpha_1$ as $\alpha$.

The main property of a quiver morphism is that it induces an $\epsilon$-bialgebra morphism on the path $\epsilon$-bialgebras.

\begin{theorem}
\label{thm:quiver}
Let $\alpha \colon Q \to Q'$ be a quiver morphism.  Then there exists a unique morphism of $\epsilon$-bialgebras $\alpha_* \colon \bk Q \to \bk Q'$ such that
\begin{enumerate}
\item
$\alpha_*|_{Q_i} = \alpha$ for $i=0,1$, and
\item
$\alpha_*(\phi_1\cdots\phi_l) = \alpha(\phi_1)\cdots\alpha(\phi_l)$ for $l \geq 2$ and $\phi_1\cdots\phi_l \in Q_l$.
\end{enumerate}
\end{theorem}

\begin{proof}
The uniqueness part is clear because $\phi_1\cdots\phi_l \in Q_l$ is the $l$-fold product of $\phi_1, \ldots , \phi_l \in Q_1$.

For the existence part, first note that $\alpha_*$ as defined in the Theorem is a well-defined linear map because $\alpha$ preserves the source and the target.  For the same reason and the assumption that $\alpha$ is injective on vertices, it follows that $\alpha_*$ is compatible with the multiplications \eqref{pathmult} in $\bk Q$ and $\bk Q'$.  To see that $\alpha_*$ is compatible with the comultiplications \eqref{pathcomult}, observe that for $\phi_1\cdots\phi_l \in Q_l$, both $\alpha_*^{\otimes 2}(\Delta(\phi_1\cdots\phi_l))$ and $\Delta(\alpha_*(\phi_1\cdots\phi_l))$ are equal to
\begin{multline*}
s(\alpha(\phi_1)) \otimes \alpha(\phi_2)\cdots\alpha(\phi_l) + \alpha(\phi_1)\cdots\alpha(\phi_{l-1}) \otimes t(\alpha(\phi_l))\\
+ \sum_{i=1}^{l-2} \alpha(\phi_1)\cdots\alpha(\phi_i) \otimes \alpha(\phi_{i+2})\cdots\phi(\phi_l)
\end{multline*}
by the definition of $\alpha_*|_{Q_l}$ and the assumption that $\alpha$ preserves the source and the target.
\end{proof}

Combining the First Twisting Principle (Theorem \ref{thm:firsttp}) and Theorem \ref{thm:quiver}, we obtain an $\epsilon$-Hom-bialgebra from each quiver and a quiver self-morphism.

\begin{corollary}
\label{cor:quiver}
Let $Q$ be a quiver and $\alpha \colon Q \to Q$ be a quiver morphism.  Then there is an $\epsilon$-Hom-bialgebra
\[
\bk Q_{\alpha_*} = (\bk Q, \mu_{\alpha_*}, \Delta_{\alpha_*}, \alpha_*),
\]
where $\alpha_* \colon \bk Q \to \bk Q$ is the induced $\epsilon$-bialgebra morphism as in Theorem \ref{thm:quiver}, $\mu_{\alpha_*} = \alpha_* \circ \mu$, and $\Delta_{\alpha_*} = \Delta \circ \alpha_*$.
\end{corollary}

Let us illustrate Corollary \ref{cor:quiver} with some examples.  To simplify the notations, we will use the same symbol for both a quiver morphism and its induced $\epsilon$-bialgebra morphism.

\begin{example}[\textbf{Kronecker $\epsilon$-Hom-bialgebras}]
\label{ex:kronecker}
In this example, we apply Corollary \ref{cor:quiver} to the \emph{Kronecker quiver} $Q$
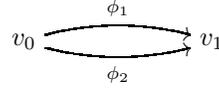
\[
\nicearrow
\xymatrix{
v_0 \ar@/^/[rr]^{\phi_1} \ar@/_/[rr]_{\phi_2} & & v_1
}
\]
with two vertices $v_i$ $(i=0,1)$ and two arrows $\phi_j$ $(j=1,2)$, both with source $v_0$ and target $v_1$.  Its path algebra is called the \emph{Kronecker algebra} and is isomorphic to the algebra of matrices
\[
\bk Q \cong \begin{bmatrix} \bk & 0\\ \bk\oplus\bk & \bk\end{bmatrix}.
\]
One way to see this is through the identification
\[
v_0 = \begin{bmatrix}0 & 0 \\ 0 & 1\end{bmatrix},\quad
v_1 = \begin{bmatrix}1 & 0 \\ 0 & 0\end{bmatrix},\quad
\phi_1 = \begin{bmatrix}0 & 0 \\ (1,0) & 0\end{bmatrix},\quad
\phi_2 = \begin{bmatrix}0 & 0 \\ (0,1) & 0\end{bmatrix}.
\]
These four elements form a basis of the path algebra $\bk Q$ because the Kronecker quiver $Q$ does not have any paths of length $\geq 2$.  The multiplication $\mu$ in the path algebra $\bk Q$ is determined by
\[
v_i^2 = v_i,\quad
v_0\phi_j = \phi_j = \phi_jv_1
\]
for $i=0,1$ and $j=1,2$; all other binary products of these basis elements are $0$.  See, e.g., (\cite{ass} page 52, Example 1.13(b)) for details.  As an $\epsilon$-bialgebra, the comultiplication in $\bk Q$ is determined by the first two lines in \eqref{pathcomult}.  Namely, for $a,b,c,d \in \bk$, we have
\begin{equation}
\label{krdelta}
\begin{split}
\Delta\krmatrix &= a\Delta\vone + b\Delta\vzero + c\Delta\phione + d\Delta\phitwo\\
&= 0 + 0 + c\vzero \otimes \vone + d\vzero \otimes \vone\\
&= (c+d)\vzero \otimes \vone.
\end{split}
\end{equation}
In particular, $\Delta$ is determined by $c+d$.

It follows from the definition of a quiver morphism that there are exactly three non-identity quiver morphisms on the Kronecker quiver $Q$, all of which have to fix both $v_i$ $(i=0,1)$.  On arrows these three morphisms are:
\begin{enumerate}
\item
$\alpha_t(\phi_j) = \phi_1$ for $j=1,2$ (where $t$ stands for \emph{top});
\item
$\alpha_b(\phi_j) = \phi_2$ for $j=1,2$ (where $b$ stands for \emph{bottom});
\item
$\alpha_p(\phi_1) = \phi_2$ and $\alpha_p(\phi_2) = \phi_1$ (where $p$ stands for \emph{permute}).
\end{enumerate}
Their induced $\epsilon$-bialgebra morphisms on $\bk Q$ are given by
\[
\alpha_t\krmatrix = \kralphat,\quad
\alpha_b\krmatrix = \kralphab,\quad
\alpha_p\krmatrix = \kralphap.
\]
By Corollary \ref{cor:quiver}, for each quiver morphism $\alpha \in \{\alpha_t,\alpha_b,\alpha_p\}$, there is a corresponding $\epsilon$-Hom-bialgebra
\[
\bk Q_{\alpha} = (\bk Q,\mualpha = \alpha \circ \mu, \Deltaalpha = \Delta \circ \alpha, \alpha).
\]
It follows from \eqref{krdelta} that in each of the three cases, we have $\Deltaalpha = \Delta$.

Note that the first derived $\epsilon$-Hom-bialgebra of $\bk Q_{\alpha_p}$ (as in the Second Twisting Principle, Theorem \ref{thm:secondtp}) is
\[
(\bk Q_{\alpha_p})^1 = (\bk Q_{\alpha_p},\alpha_p^2 \circ \mu,\Delta \circ \alpha_p^2,\alpha_p^2).
\]
Since $\alpha_p^2 = Id \colon \bk Q \to \bk Q$, it follows that $(\bk Q_{\alpha_p})^1 = \bk Q$, the original path $\epsilon$-bialgebra of the Kronecker quiver.
\qed
\end{example}

\begin{example}[\textbf{Lower-triangular matrix $\epsilon$-Hom-bialgebra}]
\label{ex:triangular}
In this example, we apply Corollary \ref{cor:quiver} to the quiver $Q$
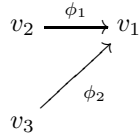
\[
\nicearrow
\xymatrix{
v_2 \ar[r]^-{\phi_1} & v_1\\
v_3 \ar[ur]_-{\phi_2} &
}
\]
with three vertices and two arrows as indicated.  Its path algebra is isomorphic to a subalgebra of the full lower-triangular matrix algebra, namely,
\[
\bk Q \cong \triangular.
\]
One way to see this is through the following identification.  Let $E_{ij}$ be the $3$-by-$3$ matrix with $1$ in the $(i,j)$-entry and $0$ everywhere else.  Then $v_i = E_{ii}$ for $i = 1,2,3$ and $\phi_j = E_{j+1,1}$ for $j=1,2$.   As in Example \ref{ex:kronecker}, the quiver $Q$ does not have any paths of length $\geq 2$.  The multiplication $\mu$ in $\bk Q$ is determined by
\[
v_i^2 = v_i,\quad \phi_j v_1 = \phi_j,\quad v_2\phi_1 = \phi_1, \quad v_3\phi_2 = \phi_2
\]
for $i=1,2,3$ and $j=1,2$; all other binary products of these basis elements are $0$.  See, e.g., (\cite{ass} page 45, Example (d)) for details.  By a computation similar to \eqref{krdelta} in Example \ref{ex:kronecker}, the comultiplication $\Delta$ in the path $\epsilon$-bialgebra $\bk Q$ is given by
\[
\Delta\qmatrix = \qmatrixxy \otimes \eone
\]
for $a,b,c,d,e \in \bk$.

There is only one non-identity quiver morphism $\alpha \colon Q \to Q$, which is given by
\[
\alpha(v_1) = v_1,\quad \alpha(v_2) = v_3,\quad \alpha(v_3) = v_2, \quad \alpha(\phi_1) = \phi_2, \quad \alpha(\phi_2) = \phi_1.
\]
Its induced $\epsilon$-bialgebra morphism on $\bk Q$ is given by
\[
\alpha \qmatrix = \qmatrixalpha.
\]
By Corollary \ref{cor:quiver}, there is a corresponding $\epsilon$-Hom-bialgebra
\[
\bk Q_{\alpha} = (\bk Q,\mualpha = \alpha \circ \mu, \Deltaalpha = \Delta \circ \alpha, \alpha).
\]
The twisted comultiplication $\Deltaalpha$ is given by
\[
\Deltaalpha \qmatrix = \qmatrixyx \otimes \eone.
\]

Observe that the first derived $\epsilon$-Hom-bialgebra of $\bk Q_\alpha$ (as in Theorem \ref{thm:secondtp}) is
\[
(\bk Q_\alpha)^1 = (\bk Q,\alpha^2 \circ \mu, \Delta \circ \alpha^2, \alpha^2).
\]
Since $\alpha^2 = Id$ on $\bk Q$, it follows that $(\bk Q_\alpha)^1$ coincides with the original path $\epsilon$-bialgebra $\bk Q$.
\qed
\end{example}

\section{Coboundary $\epsilon$-Hom-bialgebras}
\label{sec:cobound}

In this section we study a sub-class of $\epsilon$-Hom-bialgebras (Definition \ref{def:homas}) called coboundary $\epsilon$-Hom-bialgebras, which are the Hom-type analogs of Aguiar's coboundary $\epsilon$-bialgebras \cite{aguiar,aguiar2}.  The motivation for studying coboundary $\epsilon$-Hom-bialgebras is that, in an $\epsilon$-Hom-bialgebra, the comultiplication $\Delta$ is a $1$-cocycle in Hom-associative algebra cohomology (Remark \ref{rk:cocycle}).  The suitably-defined $1$-coboundaries form an obvious sub-class of the $1$-cocycles, so it makes sense to consider $\epsilon$-Hom-bialgebras whose comultiplications are $1$-coboundaries.  We prove the two Twisting Principles for coboundary $\epsilon$-Hom-bialgebras and characterize them in terms of the $A$-invariance of a certain $3$-tensor $A(r)$.  The element $A(r)$ is what defines the associative Hom-Yang-Baxter equation, which we study in section  \ref{sec:ahybe}

For a Hom-associative algebra $(A,\mu,\alpha)$, $a \in A$, and $b = \sum b_1 \otimes b_2 \in A^{\otimes 2}$, we use the notations:
\begin{equation}
\label{staraction}
\begin{split}
a * b &= \mu(a,b_1) \otimes \alpha(b_2),\\
b * a &= \alpha(b_1) \otimes \mu(b_2,a),\\
[a, b]_* &= a * b - b * a.
\end{split}
\end{equation}
In particular, we can consider the linear map $[-,b]_* \colon A \to A^{\otimes 2}$ given by $[-,b]_*(a) = [a,b]_*$.  An element in $A^{\otimes 2}$ is said to be \textbf{$\alpha$-invariant} if it is fixed by $\alpha^{\otimes 2}$.

\begin{definition}
\label{def:cobound}
A \textbf{coboundary $\epsilon$-Hom-bialgebra} $(A,\mu,\Delta,\alpha,r)$ consists of an $\epsilon$-Hom-bialgebra $(A,\mu,\Delta,\alpha)$ and an $\alpha$-invariant element $r \in A^{\otimes 2}$ such that $\Delta = [-,r]_*$. A \textbf{coboundary $\epsilon$-bialgebra} is a coboundary $\epsilon$-Hom-bialgebra with $\alpha = Id$.
\end{definition}

To be more explicit, if $r = \sum u_i \otimes v_i$, then the comultiplication $\Delta = [-,r]_*$ is given by
\begin{equation}
\label{starbracketa}
\Delta(a) = [a,r]_* = \sum \mu(a,u_i) \otimes \alpha(v_i) - \alpha(u_i) \otimes \mu(v_i,a)
\end{equation}
for $a \in A$.  In particular, a coboundary $\epsilon$-bialgebra $(A,\mu,\Delta,r)$ consists of an $\epsilon$-bialgebra $(A,\mu,\Delta)$ and a $2$-tensor $r \in A^{\otimes 2}$ such that
\[
\Delta(a) = [a,r] = ar - ra = \sum au_i \otimes v_i - u_i \otimes v_ia.
\]
In other words, in a coboundary $\epsilon$-bialgebra, the comultiplication $\Delta = [-,r]$ is the principal derivation induced by $r$.

The following two results are the two Twisting Principles for coboundary $\epsilon$-Hom-bialgebras, corresponding to Theorems \ref{thm:firsttp} and \ref{thm:secondtp}, respectively.  The first result shows that coboundary $\epsilon$-bialgebras can be twisted into coboundary $\epsilon$-Hom-bialgebras along suitable algebra morphisms.

\begin{theorem}
\label{thm:coboundtp}
Let $(A,\mu,\Delta=[-,r],r)$ be a coboundary $\epsilon$-bialgebra and $\alpha \colon A \to A$ be an algebra morphism such that $\alpha^{\otimes 2}(r) = r$. Then
\[
A_\alpha = (A,\mu_\alpha,\Delta_\alpha,\alpha,r)
\]
is a coboundary $\epsilon$-Hom-bialgebra, where $\mualpha = \alpha \circ \mu$ and $\Deltaalpha = \Delta \circ \alpha$.
\end{theorem}

\begin{proof}
It follows from the multiplicativity of $\alpha$ with respect to $\mu$ and the $\alpha$-invariance of $r$ that $\alpha^{\otimes 2} \circ \Delta = \Delta \circ \alpha$.  In other words, $\alpha$ is an $\epsilon$-bialgebra morphism. By Theorem \ref{thm:firsttp} we already know that $A_\alpha$ is an $\epsilon$-Hom-bialgebra.  It remains to show that $\Deltaalpha$ has the form $[-,r]_*$ (with $\mualpha$ instead of $\mu$).  Write $r = \sum u_i \otimes v_i$.  For any element $a \in A$, we have
\[
\begin{split}
\Deltaalpha(a) &= \alpha^{\otimes 2}(\Delta(a))\\
&= \alpha(\mu(a,u_i)) \otimes \alpha(v_i) - \alpha(u_i) \otimes \alpha(\mu(v_i,a))\\
&= \mualpha(a,u_i) \otimes \alpha(v_i) - \alpha(u_i) \otimes \mualpha(v_i,a). \end{split}
\]
This proves that $\Deltaalpha = [-,r]_*$ in $A_\alpha$.
\end{proof}

The following result is the Second Twisting Principle for coboundary $\epsilon$-Hom-bialgebras.  It shows that an arbitrary coboundary $\epsilon$-Hom-bialgebra gives rise to a derived sequence of coboundary $\epsilon$-Hom-bialgebras.

\begin{theorem}
\label{thm:coboundtp2}
Let $(A,\mu,\Delta,\alpha,r)$ be a coboundary $\epsilon$-Hom-bialgebra.  Then so is
\[
A^n = (A,\mun,\Deltan,\alpha^{2^n},r)
\]
for each $n \geq 0$, where $\mun = \alpha^{2^n-1}\circ\mu$ and $\Deltan = \Delta \circ \alpha^{2^n-1}$.
\end{theorem}

\begin{proof}
By Theorem \ref{thm:secondtp} we already know that $A^n$ is an $\epsilon$-Hom-bialgebra.  Moreover, since $r$ is fixed by $\alpha^{\otimes 2}$, it is also fixed by $(\alpha^{2^n})^{\otimes 2}$.  Thus, it remains to show that $\Deltan$ has the form $[-,r]_*$ (with $\mun$ and $\alpha^{2^n}$ instead of $\mu$ and $\alpha$). Writing $r = \sum u_i \otimes v_i$, for an arbitrary element $a \in A$, we have:
\[
\begin{split}
\Deltan(a) &= \Delta(\alpha^{2^n-1}(a))\\
&= \mu(\alpha^{2^n-1}(a),u_i) \otimes \alpha(v_i) - \alpha(u_i) \otimes \mu(v_i,\alpha^{2^n-1}(a))\\
&= \mu(\alpha^{2^n-1}(a),\alpha^{2^n-1}(u_i)) \otimes \alpha^{2^n}(v_i) - \alpha^{2^n}(u_i) \otimes \mu(\alpha^{2^n-1}(v_i),\alpha^{2^n-1}(a))\\
&= \mun(a,u_i) \otimes \alpha^{2^n}(v_i) - \alpha^{2^n}(u_i) \otimes \mun(v_i,a).
\end{split}
\]
In the third equality above, we used the $\alpha$-invariance of $r$.  In the fourth equality, we used $\alpha\circ\mu = \mu\circ\alpha^{\otimes 2}$.  We have shown that $A^n$ is a coboundary $\epsilon$-Hom-bialgebra.
\end{proof}

In a coboundary $\epsilon$-Hom-bialgebra, the comultiplication $\Delta = [-,r]_*$ is determined by the $\alpha$-invariant $2$-tensor $r \in A^{\otimes 2}$.  It is therefore natural to ask the following question:
\begin{quote}
If $(A,\mu,\alpha)$ is a Hom-associative algebra and $r \in A^{\otimes 2}$ is $\alpha$-invariant, under what conditions is $(A,\mu,\Delta=[-,r]_*,\alpha,r)$ a coboundary $\epsilon$-Hom-bialgebra?
\end{quote}
To state the answer to this question, we need the following notations.  Let $(A,\mu,\alpha)$ be a Hom-associative algebra and $r = \sum u_i \otimes v_i, s = \sum u_j' \otimes v_j' \in A^{\otimes 2}$.  Define the $3$-tensors:
\begin{equation}
\label{ar}
\begin{split}
r_{13}s_{12} &= \sum \mu(u_i,u_j') \otimes \alpha(v_j') \otimes \alpha(v_i),\\
r_{12}s_{23} &= \sum \alpha(u_i) \otimes \mu(v_i,u_j') \otimes \alpha(v_j'),\\
r_{23}s_{13} &= \sum \alpha(u_j') \otimes \alpha(u_i) \otimes \mu(v_i,v_j'),\\
A(r) &= r_{13}r_{12} - r_{12}r_{23} + r_{23}r_{13}.
\end{split}
\end{equation}
Recall the $\bullet$-action of $A$ on $A^{\otimes n}$ defined in \eqref{dotaction}.  Define the $\bullet$-bracket
\begin{equation}
\label{dotbracket}
\begin{split}
[a,b]_\bullet &= a \bullet b - b \bullet a\\
&= \mu(\alpha(a),b_1) \otimes \alpha(b_2) \otimes \cdots \otimes \alpha(b_n) - \alpha(b_1) \otimes \cdots \otimes \alpha(b_{n-1}) \otimes \mu(b_n,\alpha(a))
\end{split}
\end{equation}
for $a \in A$ and $b = \sum b_1 \otimes \cdots \otimes b_n \in A^{\otimes n}$.  An $n$-tensor $b \in A^{\otimes n}$ is said to be \textbf{$A$-invariant} if $[a,b]_\bullet = 0$ for all $a \in A$.  The following result, which is a Hom-type generalization of (\cite{aguiar} Proposition 5.1), gives an answer to the question above.

\begin{theorem}
\label{thm:coboundchar}
Let $(A,\mu,\alpha)$ be a Hom-associative algebra and $r \in A^{\otimes 2}$ be an $\alpha$-invariant $2$-tensor.  Define the map $\Delta = [-,r]_* \colon A \to A^{\otimes 2}$ as in \eqref{staraction}.  Then $(A,\mu,\Delta,\alpha,r)$ is a coboundary $\epsilon$-Hom-bialgebra if and only if $A(r) \in A^{\otimes 3}$ is $A$-invariant.
\end{theorem}

\begin{proof}
First note that $(A,\mu,\Delta,\alpha,r)$ is a coboundary $\epsilon$-Hom-bialgebra if and only if (i) $(A,\Delta,\alpha)$ is a Hom-coassociative coalgebra \eqref{homcoassaxioms} and (ii) the cocycle condition \eqref{cocycle} is satisfied.  The comultiplicativity of $\alpha$ with respect to $\Delta = [-,r]_*$ is automatic because (with $r = \sum u_i \otimes v_i$)
\begin{equation}
\label{deltaalphacommute}
\begin{split}
\Delta(\alpha(a)) &= \alpha(a)u_i \otimes \alpha(v_i) - \alpha(u_i) \otimes v_i\alpha(a)\\
&= \alpha(a)\alpha(u_i) \otimes \alpha^2(v_i) - \alpha^2(u_i) \otimes \alpha(v_i)\alpha(a)\\
&= \alpha^{\otimes 2}(au_i \otimes \alpha(v_i) - \alpha(u_i) \otimes v_ia)\\
&= \alpha^{\otimes 2}(\Delta(a)).
\end{split}
\end{equation}
In the second equality above, we used the $\alpha$-invariance of $r$.

In Lemma \ref{lem:deltacocycle} below, we will show that $\Delta = [-,r]_*$ satisfies the cocycle condition \eqref{cocycle}, regardless of whether $A(r)$ is $A$-invariant or not.  Moreover, in Lemma \ref{lem:coboundchar} below, we will show that the Hom-coassociativity of $\Delta=[-,r]_*$ is equivalent to the $A$-invariance of the $3$-tensor $A(r)$.  Therefore, to finish the proof of the Theorem, it remains to prove Lemmas \ref{lem:deltacocycle} and \ref{lem:coboundchar}.
\end{proof}

\begin{lemma}
\label{lem:deltacocycle}
With the same hypotheses as in Theorem \ref{thm:coboundchar}, the cocycle condition \eqref{cocycle} is satisfied for $\Delta = [-,r]_*$.
\end{lemma}

\begin{proof}
For elements $a,b \in A$, the left-hand side of the cocycle condition (in the equivalent form \eqref{cocycle'}) is:
\begin{equation}
\label{deltaab}
\begin{split}
\Delta(ab) &= [ab,r]_*\\
&= (ab)u_i \otimes \alpha(v_i) - \alpha(u_i) \otimes v_i(ab)\\
&= (ab)\alpha(u_i) \otimes \alpha^2(v_i) - \alpha^2(u_i) \otimes \alpha(v_i)(ab)\\
&= \alpha(a)(bu_i) \otimes \alpha^2(v_i) - \alpha^2(u_i) \otimes (v_ia)\alpha(b).
\end{split}
\end{equation}
In the third and the fourth equalities above, we used the $\alpha$-invariance of $r$ and the Hom-associativity axiom \eqref{homassaxioms}, respectively.  On the other hand, we have
\[
\Delta(a) = \sum a_1 \otimes a_2 = \sum au_i \otimes \alpha(v_i) - \alpha(u_i) \otimes v_ia,
\]
and similarly for $\Delta(b) = \sum b_1 \otimes b_2$.  Therefore, the right-hand side of the cocycle condition \eqref{cocycle'} is:
\[
\begin{split}
\alpha(a)b_1 \otimes \alpha(b_2) + \alpha(a_1) \otimes a_2 \alpha(b) &=
\alpha(a)(bu_i) \otimes \alpha^2(v_i) - \alpha(a)\alpha(u_i) \otimes \alpha(v_ib)\\
&\relphantom{} + \alpha(au_i) \otimes \alpha(v_i)\alpha(b) - \alpha^2(u_i) \otimes (v_ia)\alpha(b).
\end{split}
\]
On the right-hand side, the second and the third terms cancel out by the multiplicativity of $\alpha$.  The sum of the other two terms is $\Delta(ab)$ \eqref{deltaab}.  We have shown that the cocycle condition \eqref{cocycle} holds.
\end{proof}

\begin{lemma}
\label{lem:coboundchar}
With the same hypotheses as in Theorem \ref{thm:coboundchar}, we have
\[
(\alpha \otimes \Delta) \circ \Delta - (\Delta \otimes \alpha) \circ \Delta = -[-,A(r)]_\bullet,
\]
where $[-,-]_\bullet$ is defined in \eqref{dotbracket}.  In particular, $\Delta = [-,r]_*$ is Hom-coassociative \eqref{homcoassaxioms} if and only if $A(r)$ is $A$-invariant.
\end{lemma}

\begin{proof}
With $r = \sum u_i \otimes v_i$, the left-hand side of the Hom-coassociativity axiom \eqref{homcoassaxioms} for $\Delta = [-,r]_*$ applied to an element $a \in A$ is:
\[
\begin{split}
(\alpha \otimes \Delta)(\Delta(a))
&= (\alpha \otimes \Delta)(au_i \otimes \alpha(v_i) - \alpha(u_i) \otimes v_ia)\\
&= \alpha(a)\alpha(u_i) \otimes \alpha^{\otimes 2}(\Delta(v_i)) - \alpha^2(u_i) \otimes \Delta(v_ia) \quad\text{by \eqref{deltaalphacommute}}\\ &= \alpha(a)\alpha(u_i) \otimes \alpha(v_iu_j) \otimes \alpha^2(v_j) - \alpha(a)\alpha(u_i) \otimes \alpha^2(u_j) \otimes \alpha(v_jv_i)\\
&\relphantom{} - \alpha^2(u_i) \otimes \alpha(v_i)(au_j) \otimes \alpha^2(v_j) + \alpha^2(u_i) \otimes \alpha^2(u_j) \otimes (v_jv_i)\alpha(a) \quad \text{by \eqref{deltaab}}\\
&= a \bullet (r_{12}r_{23}) - a \bullet (r_{23}r_{13}) - \alpha^2(u_i) \otimes (v_ia)\alpha(u_j) \otimes \alpha^2(v_j) + (r_{23}r_{13}) \bullet a.
\end{split}
\]
Recall that the $3$-tensors $r_{12}r_{23}$, etc., are defined in \eqref{ar} and that the $\bullet$-action is defined in \eqref{dotaction}.  In the last equality above, we used the Hom-associativity \eqref{homassaxioms} of $\mu$ in the third summand.

Similarly, the right-hand side of the Hom-coassociativity axiom \eqref{homcoassaxioms} applied to $a \in A$ is:
\[
\begin{split}
(\Delta \otimes \alpha)(\Delta(a))
&= (\Delta \otimes \alpha)(au_i \otimes \alpha(v_i) - \alpha(u_i) \otimes v_ia)\\
&= \Delta(au_i) \otimes \alpha^2(v_i) - \alpha^{\otimes 2}(\Delta(u_i)) \otimes \alpha(v_i)\alpha(a)\\
&= \alpha(a)(u_iu_j) \otimes \alpha^2(v_j) \otimes \alpha^2(v_i) - \alpha^2(u_j) \otimes (v_ja)\alpha(u_i) \otimes \alpha^2(v_i)\\
&\relphantom{} - \alpha(u_iu_j) \otimes \alpha^2(v_j) \otimes \alpha(v_i)\alpha(a) + \alpha^2(u_j) \otimes \alpha(v_ju_i) \otimes \alpha(v_i)\alpha(a)\\
&= a \bullet (r_{13}r_{12}) - \alpha^2(u_j) \otimes (v_ja)\alpha(u_i) \otimes \alpha^2(v_i) - (r_{13}r_{12}) \bullet a + (r_{12}r_{23}) \bullet a.
\end{split}
\]
Combining the computations above, we have
\[
\begin{split}
(\alpha \otimes \Delta)(\Delta(a)) - (\Delta \otimes \alpha)(\Delta(a))
&= a \bullet (r_{12}r_{23}) - a \bullet (r_{23}r_{13}) + (r_{23}r_{13}) \bullet a\\
&\relphantom{} - a \bullet (r_{13}r_{12}) + (r_{13}r_{12}) \bullet a - (r_{12}r_{23}) \bullet a\\
&= -[a,A(r)]_\bullet.
\end{split}
\]
This proves the first assertion.  The second assertion is an immediate consequence of the first assertion because the $A$-invariance of $A(r)$ means that $[a,A(r)]_\bullet = 0$ for all $a \in A$.
\end{proof}

\section{Perturbation of comultiplication}
\label{sec:pertubation}

The purpose of this section is to discuss perturbation of the comultiplication in an $\epsilon$-Hom-bialgebra by suitable coboundaries, along the lines of Drinfel'd's perturbation theory of quasi-Hopf algebras \cite{dri83b,dri89b,dri90,dri91,dri92}.  Recall from Remark \ref{rk:cocycle} that the condition \eqref{cocycle} says that in an $\epsilon$-Hom-bialgebra $(A,\mu,\Delta,\alpha)$, the comultiplication $\Delta \colon A \to A^{\otimes 2}$ is a $1$-cocycle (in Hom-associative algebra cohomology with coefficients in $A^{\otimes 2}$).  In analogy with Hochschild cohomology, we can define the $0$-cochains $C^0(A,A^{\otimes 2})$ as the space $\{b \in A^{\otimes 2} \colon \alpha^{\otimes 2}(b) = b\}$ of $\alpha$-invariant $2$-tensors and the differential $\delta^0 \colon C^0(A,A^{\otimes 2}) \to C^1(A,A^{\otimes 2})$ as $\delta^0(r) = [-,r]_*$ \eqref{staraction}.  It is indeed the case that $\delta^1 \circ \delta^0 = 0$ (with $\delta^1$ as in \eqref{delta1}) because by Lemma \ref{lem:deltacocycle}, $[-,r]_*$ satisfies the cocycle condition \eqref{cocycle}.  In other words, $[-,r]_*$ is a $1$-coboundary, provided $r \in A^{\otimes 2}$ is $\alpha$-invariant.

Perturbation of cocycles by coboundaries is a natural construction in homological algebra.  Therefore, it makes sense to ask the following question:
\begin{quote}
Let $(A,\mu,\Delta,\alpha)$ be an $\epsilon$-Hom-bialgebra and $r \in A^{\otimes 2}$ be an $\alpha$-invariant $2$-tensor.  Define the perturbed comultiplication $\Delta' = \Delta + \Delta_r \colon A \to A^{\otimes 2}$, where $\Delta_r = [-,r]_*$.  Under what conditions is $A' = (A,\mu,\Delta',\alpha)$ an $\epsilon$-Hom-bialgebra?
\end{quote}
To state the answer to this question, we use the following shorthand:
\begin{equation}
\label{adminus}
(\alpha \otimes \Delta)^- = \alpha \otimes \Delta - \Delta \otimes \alpha.
\end{equation}
Also recall the $3$-tensors $r_{23}s_{13}$, $r_{13}s_{12}$, and $A(r)$ \eqref{ar}, and the bracket $[-,-]_\bullet$ \eqref{dotbracket}.  The following result gives an answer to the question above.

\begin{theorem}
\label{thm:perturb}
Let $(A,\mu,\Delta,\alpha)$ be an $\epsilon$-Hom-bialgebra and $r \in A^{\otimes 2}$ be an $\alpha$-invariant $2$-tensor.  Define the perturbed comultiplication $\Delta' = \Delta + \Delta_r \colon A \to A^{\otimes 2}$, where $\Delta_r = [-,r]_*$.  Then $A' = (A,\mu,\Delta',\alpha)$ is an $\epsilon$-Hom-bialgebra if and only if
\begin{equation}
\label{dagga'}
[a, (\alpha \otimes \Delta)^-(r) - A(r)]_\bullet = r_{23}\Delta(a)_{13} + \Delta(a)_{13}r_{12}
\end{equation}
for all $a \in A$.
\end{theorem}

Before we give the proof of Theorem \ref{thm:perturb}, let us discuss some of its special cases.  First, restricting to $\alpha = Id$, Theorem \ref{thm:perturb} gives a necessary and sufficient condition (\eqref{dagga'} with $\alpha = Id$) under which an $\epsilon$-bialgebra $(A,\mu,\Delta)$ gives rise to a perturbed $\epsilon$-bialgebra $A' = (A,\mu,\Delta' = \Delta + \Delta_r)$.

Second, in Theorem \ref{thm:perturb}, the special case $\Delta = 0$ is Theorem \ref{thm:coboundchar}.  Indeed, when $\Delta = 0$, the condition \eqref{dagga'} reduces to
\[
[a,-A(r)]_\bullet = 0
\]
for all $a \in A$, which is equivalent to $A(r)$ being $A$-invariant.

\begin{proof}[Proof of Theorem \ref{thm:perturb}]
First note that $A'$ is an $\epsilon$-Hom-bialgebra if and only if (i) $(A,\Delta',\alpha)$ is a Hom-coassociative coalgebra and (ii) the cocycle condition \eqref{cocycle} is satisfied for $\Delta'$.  Since both $\Delta$ and $\Delta_r$ commute with $\alpha$ (see \eqref{deltaalphacommute} for the latter), so does $\Delta' = \Delta + \Delta_r$.  Likewise, since both $\Delta$ and $\Delta_r$ satisfy the cocycle condition \eqref{cocycle} (see Lemma \ref{lem:deltacocycle} for the latter), so does their sum $\Delta'$.  In Lemma \ref{lem:perturb} below, we will show that $\Delta'$ is Hom-coassociative if and only if \eqref{dagga'} holds.  Therefore, to finish the proof of the Theorem, it remains to establish the following Lemma.
\end{proof}

\begin{lemma}
\label{lem:perturb}
With the hypotheses of Theorem \ref{thm:perturb}, the perturbed comultiplication $\Delta'$ is Hom-coassociative if and only if \eqref{dagga'} is satisfied.
\end{lemma}

\begin{proof}
Since $\Delta' = \Delta + \Delta_r$, the left-hand side of the Hom-coassociativity axiom \eqref{homcoassaxioms} for $\Delta'$ is
\[
\begin{split}
(\alpha \otimes \Delta') \circ \Delta'
&= (\alpha \otimes \Delta) \circ \Delta + (\alpha \otimes \Delta) \circ \Delta_r\\
&\relphantom{} + (\alpha \otimes \Delta_r) \circ \Delta + (\alpha \otimes \Delta_r) \circ \Delta_r.
\end{split}
\]
Likewise, the right-hand side of the Hom-coassociativity axiom for $\Delta'$ is
\[
\begin{split}
(\Delta' \otimes \alpha) \circ \Delta'
&= (\Delta \otimes \alpha) \circ \Delta + (\Delta \otimes \alpha) \circ \Delta_r\\
&\relphantom{} (\Delta_r \otimes \alpha) \circ \Delta + (\Delta_r \otimes \alpha) \circ \Delta_r.
\end{split}
\]
By assumption $\Delta$ is Hom-coassociative, i.e., $(\alpha \otimes \Delta) \circ \Delta = (\Delta \otimes \alpha) \circ \Delta$.  Thus, the computations above imply that $\Delta'$ is Hom-coassociative if and only if
\begin{equation}
\label{dagga}
(\alpha \otimes \Delta - \Delta \otimes \alpha) \circ \Delta_r + (\alpha \otimes \Delta_r - \Delta_r \otimes \alpha) \circ \Delta
= (\Delta_r \otimes \alpha - \alpha \otimes \Delta_r) \circ \Delta_r.
\end{equation}
By Lemma \ref{lem:coboundchar} the right-hand side of \eqref{dagga} is equal to $[-,A(r)]_\bullet$.  Thus, using the shorthand \eqref{adminus}, to show the equivalence between \eqref{dagga'} and \eqref{dagga}, it is enough to show
\begin{equation}
\label{dagga2}
\begin{split}
(\alpha \otimes \Delta)^-(\Delta_r(a)) & + (\alpha \otimes \Delta_r)^-(\Delta(a))\\
&= [a, (\alpha \otimes \Delta)^-(r)]_\bullet - r_{23}\Delta(a)_{13} - \Delta(a)_{13}r_{12}
\end{split}
\end{equation}
for all $a \in A$.

Write $r = \sum u_i \otimes v_i$.  Since $\Delta_r(a) = au_i \otimes \alpha(v_i) - \alpha(u_i) \otimes v_ia$, using the notation in \eqref{cocycle''}, the first summand on the left-hand side of \eqref{dagga2} is:
\begin{equation}
\label{dagga2l}
\begin{split}
&(\alpha \otimes \Delta)^-(\Delta_r(a))\\
&= \alpha(a)\alpha(u_i) \otimes \Delta(\alpha(v_i)) - \alpha^2(u_i) \otimes \Delta(v_ia) - \Delta(au_i) \otimes \alpha^2(v_i) + \Delta(\alpha(u_i)) \otimes \alpha(v_i)\alpha(a)\\
&= \alpha(a)\alpha(u_i) \otimes \alpha^{\otimes 2}(\Delta(v_i)) - \alpha^2(u_i) \otimes v_i \bullet \Delta(a) - \alpha^2(u_i) \otimes \Delta(v_i) \bullet a\\
&\relphantom{} - a \bullet \Delta(u_i) \otimes \alpha^2(v_i) - \Delta(a) \bullet u_i \otimes \alpha^2(v_i) + \alpha^{\otimes 2}(\Delta(u_i)) \otimes \alpha(v_i)\alpha(a)\\
&= a \bullet ((\alpha \otimes \Delta)(r)) - r_{12}\Delta(a)_{23} - ((\alpha \otimes \Delta)(r))\bullet a\\
&\relphantom{} - a \bullet ((\Delta \otimes \alpha)(r)) - \Delta(a)_{12}r_{23} + ((\Delta \otimes \alpha)(r)) \bullet a\\
&= [a, (\alpha \otimes \Delta)^-(r)]_\bullet - r_{12}\Delta(a)_{23} - \Delta(a)_{12}r_{23}.
\end{split}
\end{equation}
In the third equality above, we used the $\alpha$-invariance of $r$ on the second and the fifth summands.  Likewise, writing $\Delta(a) = \sum a_1 \otimes a_2$, the second summand on the left-hand side of \eqref{dagga2} is:
\begin{equation}
\label{dagga2r}
\begin{split}
(\alpha \otimes \Delta_r)^-(\Delta(a))
&= \alpha(a_1) \otimes \Delta_r(a_2) - \Delta_r(a_1) \otimes \alpha(a_2)\\
&= \alpha(a_1) \otimes a_2u_i \otimes \alpha(v_i) - \alpha(a_1) \otimes \alpha(u_i) \otimes v_ia_2 \\
&\relphantom{} - a_1u_i \otimes \alpha(v_i) \otimes \alpha(a_2) + \alpha(u_i) \otimes v_ia_1 \otimes \alpha(a_2)\\
&= \Delta(a)_{12}r_{23} - r_{23}\Delta(a)_{13} - \Delta(a)_{13}r_{12} + r_{12}\Delta(a)_{23}.
\end{split}
\end{equation}
Combining \eqref{dagga2l} and \eqref{dagga2r}, we obtain the desired equality \eqref{dagga2}.
\end{proof}

\section{The associative Hom-Yang-Baxter equation and quasi-triangular $\epsilon$-Hom-bialgebras}
\label{sec:ahybe}

The purpose of this section is to study a sub-class of the coboundary $\epsilon$-Hom-bialgebras (Definition \ref{def:cobound}), called quasi-triangular $\epsilon$-Hom-bialgebras, and the closely related associative Hom-Yang-Baxter equation.  Recall from Theorem \ref{thm:coboundchar} that, given a Hom-associative algebra $(A,\mu,\alpha)$ and an $\alpha$-invariant $2$-tensor $r \in A^{\otimes 2}$, the tuple $(A,\mu,\Delta=[-,r]_*,\alpha,r)$ is a coboundary $\epsilon$-Hom-bialgebra if and only if the $3$-tensor $A(r)$ \eqref{ar} is $A$-invariant.  Since $0$ is the simplest $A$-invariant element, this leads us naturally to the following definitions.

\begin{definition}
\label{def:ahybe}
Let $(A,\mu,\alpha)$ be a Hom-associative algebra and $r \in A^{\otimes 2}$.  We say that $r$ is a solution of the \textbf{associative Hom-Yang-Baxter equation} (AHYBE) in $A$ if
\begin{equation}
\label{ahybe}
A(r) = r_{13}r_{12} - r_{12}r_{23} + r_{23}r_{13} = 0,
\end{equation}
where the $3$-tensors $r_{13}r_{12}$, etc., are defined in \eqref{ar}.  Define a \textbf{quasi-triangular $\epsilon$-Hom-bialgebra} as a coboundary $\epsilon$-Hom-bialgebra $(A,\mu,\Delta,\alpha,r)$ in which $r$ is a solution of the AHYBE.  In the special case that $A$ is an associative algebra (i.e., $\alpha = Id$), we obtain the definitions of the \textbf{associative Yang-Baxter equation} (AYBE) and of a \textbf{quasi-triangular $\epsilon$-bialgebra}.
\end{definition}

Note that a quasi-triangular $\epsilon$-bialgebra is by definition a coboundary $\epsilon$-bialgebra $(A,\mu,\Delta,r)$ in which $r$ is a solution of the AYBE.  The AYBE and quasi-triangular $\epsilon$-bialgebras were introduced by Aguiar \cite{aguiar,aguiar2}.  The AYBE is the associative version of the classical Yang-Baxter equation (CYBE) \cite{skl1,skl2}.  Indeed, Aguiar (\cite{aguiar2} Theorem 3.5 and Corollary 3.7) showed that, under the assumption that the symmetric part $r^+$ of $r$ is $A$-invariant, a solution of the AYBE gives rise to a solution of the CYBE, and a quasi-triangular $\epsilon$-bialgebra gives rise to a Drinfel'd quasi-triangular Lie-bialgebra \cite{dri83,dri87}.  The AYBE with spectral parameters was introduced by Polishchuk \cite{pol} and was studied further by Schedler \cite{schedler}.

In this section, we prove the two Twisting Principles for solutions of the AHYBE and for quasi-triangular $\epsilon$-Hom-bialgebras.  We also give several characterizations of the AHYBE in a coboundary $\epsilon$-Hom-bialgebra.  The relationships between the AHYBE and the classical Hom-Yang-Baxter equation (CHYBE) \cite{yau8} and between quasi-triangular $\epsilon$-Hom-bialgebras and Hom-Lie bialgebras will be studied in section \ref{sec:qt}.  As we will see in that section, just as the AYBE is the associative analog of the CYBE, the AHYBE is the Hom-associative analog of the CHYBE.

We begin with the First Twisting Principle for the AHYBE.  It says that a solution of the AYBE and an algebra morphism give rise to a derived sequence of solutions of the AHYBE.

\begin{theorem}
\label{thm:ahybe1}
Let $(A,\mu)$ be an associative algebra, $r \in A^{\otimes 2}$ be a solution of the AYBE, and $\alpha \colon A \to A$ be an algebra morphism.  Then $(\alpha^{\otimes 2})^n(r)$ is a solution of the AHYBE in the Hom-associative algebra $A_\alpha = (A,\mualpha = \alpha\circ\mu,\alpha)$ for each $n \geq 0$.
\end{theorem}

\begin{proof}
That $A_\alpha$ is a Hom-associative algebra was proved in \cite{yau2} (Theorem 2.3), and we already used this fact in the proof of Theorem \ref{thm:firsttp}.  Let us write $r^n$ for $(\alpha^{\otimes 2})^n(r)$ and $A_\alpha(\cdot) = 0$ for the AHYBE in $A_\alpha$.  To prove the Theorem, it suffices to show that
\begin{equation}
\label{aalphar}
A_\alpha(r^n) = (\alpha^{\otimes 3})^{n+1}(A(r)),
\end{equation}
since by assumption $A(r) = 0$.  Writing $r = \sum u_i \otimes v_i$, we have $r^n = \sum \alpha^n(u_i) \otimes \alpha^n(v_i)$.  Note that
\[
\mualpha \circ (\alpha^{\otimes 2})^n = \alpha \circ \mu \circ (\alpha^{\otimes 2})^n = \alpha^{n+1} \circ \mu
\]
by the multiplicativity of $\alpha$ \eqref{homassaxioms}.  With the definitions in \eqref{ar}, we compute as follows:
\[
\begin{split}
A_\alpha(r^n) &= \mualpha(\alpha^n(u_i),\alpha^n(u_j)) \otimes \alpha(\alpha^n(v_j)) \otimes \alpha(\alpha^n(v_i))\\
&\relphantom{} - \alpha(\alpha^n(u_i)) \otimes \mualpha(\alpha^n(v_i),\alpha^n(u_j)) \otimes \alpha(\alpha^n(v_j))\\
&\relphantom{} + \alpha(\alpha^n(u_i)) \otimes \alpha(\alpha^n(u_j)) \otimes \mu_\alpha(\alpha^n(v_j),\alpha^n(v_i))\\
&= (\alpha^{\otimes 3})^{n+1}(\mu(u_i,u_j) \otimes v_j \otimes v_i - u_i \otimes \mu(v_i,u_j) \otimes v_j + u_i \otimes u_j \otimes \mu(v_j,v_i))\\
&= (\alpha^{\otimes 3})^{n+1}(A(r)).
\end{split}
\]
This proves \eqref{aalphar}, as desired.
\end{proof}

\begin{example}
\label{ex:unitalA}
Let $(A,\mu)$ be a unital associative algebra and $a \in A$ be a square-zero element, i.e., $a^2 = 0$.  Then $r = 1 \otimes a$ is a solution of the AYBE; see \cite{aguiar} Example 5.4.1.  Let $\alpha \colon A \to A$ be an algebra morphism, not-necessarily preserving the unit.  Then, by Theorem \ref{thm:ahybe1}, $A_\alpha = (A,\mualpha = \alpha \circ \mu,\alpha)$ is a Hom-associative algebra, and
\[
(\alpha^{\otimes 2})^n(r) = \alpha^n(1) \otimes \alpha^n(a)
\]
is a solution of the AHYBE \eqref{ahybe} in $A_\alpha$ for each $n \geq 0$.\qed
\end{example}

\begin{example}
\label{ex:nonunitalA}
Let $(A,\mu)$ be a not-necessarily unital associative algebra and $b \in A$ be a square-zero element.  Then $r = b \otimes b$ is a solution of the AYBE; see \cite{aguiar} Example 5.4.4.  Let $\alpha \colon A \to A$ be an algebra morphism.  Then, by Theorem \ref{thm:ahybe1}, $A_\alpha = (A,\mualpha=\alpha\circ\mu,\alpha)$ is a Hom-associative algebra, and
\[
(\alpha^{\otimes 2})^n(r) = \alpha^n(b) \otimes \alpha^n(b)
\]
is a solution of the AHYBE \eqref{ahybe} in $A_\alpha$ for each $n \geq 0$.
\qed
\end{example}

The following result is the Second Twisting Principle for the AHYBE.  It says that every solution $r$ of the AHYBE in a Hom-associative algebra $A$ is also a solution of the AHYBE in the derived Hom-associative algebras $A^n$ (Theorem \ref{thm:secondtp}).

\begin{theorem}
\label{thm:ahybe2}
Let $(A,\mu,\alpha)$ be a Hom-associative algebra and $r \in A^{\otimes 2}$ be a solution of the AHYBE.  Then $r$ is also a solution of the AHYBE in the Hom-associative algebra $A^n = (A,\mun,\alpha^{2^n})$ for each $n \geq 0$, where $\mun = \alpha^{2^n-1} \circ \mu$.
\end{theorem}

\begin{proof}
One can check directly that the axioms \eqref{homassaxioms} hold for $A^n$.  Let us write $A^n(\cdot) = 0$ for the AHYBE in $A^n$.  To prove the Theorem, it suffices to show that
\begin{equation}
\label{anr}
A^n(r) = (\alpha^{\otimes 3})^{2^n-1}(A(r)),
\end{equation}
since by assumption $A(r) = 0$.  We compute as follows:
\[
\begin{split}
A^n(r) &= \mun(u_i,u_j) \otimes \alpha^{2^n}(v_j) \otimes \alpha^{2^n}(v_i)
- \alpha^{2^n}(u_i) \otimes \mun(v_i,u_j) \otimes \alpha^{2^n}(v_j)\\
&\relphantom{} + \alpha^{2^n}(u_i) \otimes \alpha^{2^n}(u_j) \otimes \mun(v_j,v_i)\\
&= (\alpha^{\otimes 3})^{2^n-1}(\mu(u_i,u_j) \otimes \alpha(v_j) \otimes \alpha(v_i)) - (\alpha^{\otimes 3})^{2^n-1}(\alpha(u_i) \otimes \mu(v_i,u_j) \otimes \alpha(v_j))\\
&\relphantom{} + (\alpha^{\otimes 3})^{2^n-1}(\alpha(u_i) \otimes \alpha(u_j) \otimes \mu(v_j,v_i))\\
&= (\alpha^{\otimes 3})^{2^n-1}(A(r)).
\end{split}
\]
This proves \eqref{anr}, as desired.
\end{proof}

We now combine the results about coboundary $\epsilon$-Hom-bialgebras and the AHYBE to obtain results about quasi-triangular $\epsilon$-Hom-bialgebras.  The following result says that a quasi-triangular $\epsilon$-Hom-bialgebra can be constructed from a Hom-associative algebra and an $\alpha$-invariant solution of the AHYBE.

\begin{corollary}
\label{cor:qtchar}
Let $(A,\mu,\alpha)$ be a Hom-associative algebra and $r \in A^{\otimes 2}$.  Then $(A,\mu,\Delta=[-,r]_*,\alpha,r)$ is a quasi-triangular $\epsilon$-Hom-bialgebra if and only if $r$ is an $\alpha$-invariant solution of the AHYBE.
\end{corollary}

\begin{proof}
The ``only if" part follows immediately from the definition of a quasi-triangular $\epsilon$-Hom-bialgebra.  The ``if" part follows from Theorem \ref{thm:coboundchar} because $A(r) = 0$ is trivially $A$-invariant.
\end{proof}

The next two results are the Twisting Principles for quasi-triangular $\epsilon$-Hom-bialgebras.

\begin{corollary}
\label{cor:qt1}
Let $(A,\mu,\Delta,r)$ be a quasi-triangular $\epsilon$-bialgebra and $\alpha \colon A \to A$ be an algebra morphism such that $\alpha^{\otimes 2}(r) = r$. Then
\[
A_\alpha = (A,\mualpha,\Deltaalpha,\alpha,r)
\]
is a quasi-triangular $\epsilon$-Hom-bialgebra, where $\mualpha = \alpha\circ\mu$ and $\Deltaalpha = \Delta\circ\alpha$.
\end{corollary}

\begin{proof}
That $A_\alpha$ is a coboundary $\epsilon$-Hom-bialgebra follows from Theorem \ref{thm:coboundtp}.  That $r$ is a solution of the AHYBE in $A_\alpha$ follows from the $n=0$ case of Theorem \ref{thm:ahybe1}.
\end{proof}

\begin{corollary}
\label{cor:qt2}
Let $(A,\mu,\Delta,\alpha,r)$ be a quasi-triangular $\epsilon$-Hom-bialgebra. Then so is
\[
A^n = (A, \mun, \Deltan, \alpha^{2^n}, r)
\]
for each $n \geq 0$, where $\mun = \alpha^{2^n - 1}\circ\mu$ and $\Deltan = \Delta \circ \alpha^{2^n-1}$.
\end{corollary}

\begin{proof}
This follows immediately from the Second Twisting Principles for coboundary $\epsilon$-Hom-bialgebras (Theorem \ref{thm:coboundtp2}) and for the AHYBE (Theorem \ref{thm:ahybe2}).
\end{proof}

\begin{example}
\label{ex:unitalA'}
This is a continuation of Example \ref{ex:unitalA}, in which $(A,\mu)$ is a unital associative algebra, $a\in A$ is a square-zero element, and $r = 1 \otimes a$.  There is a quasi-triangular $\epsilon$-bialgebra $(A,\mu,\Delta,r)$ in which
\[
\Delta(b) = [b,r] = b \otimes a - 1 \otimes ab
\]
for $b \in A$ (\cite{aguiar} Example 5.4.1).  Let $\alpha \colon A \to A$ be an algebra morphism such that $\alpha^{\otimes 2}(r) = r$ (e.g., $\alpha(1) = 1$ and $\alpha(a) = a$).  By Corollary \ref{cor:qt1} there is a quasi-triangular $\epsilon$-Hom-bialgebra
\[
A_\alpha = (A,\mualpha=\alpha\circ\mu,\Deltaalpha= \Delta\circ\alpha,\alpha,r).
\]
More explicitly, we have
\[
\Deltaalpha(b) = \alpha(b) \otimes a - 1 \otimes a\alpha(b)
\]
for $b \in A$.
\qed
\end{example}

\begin{example}
\label{ex:nonunitalA'}
This is a continuation of Example \ref{ex:nonunitalA}, in which $(A,\mu)$ is a not-necessarily unital associative algebra, $b \in A$ is a square-zero element, and $r = b \otimes b$.  There is a quasi-triangular $\epsilon$-bialgebra $(A,\mu,\Delta,r)$ in which
\[
\Delta(a) = [a,r] = ab \otimes b - b \otimes ba
\]
for $a \in A$ (\cite{aguiar} Example 5.4.4).  Let $\alpha \colon A \to A$ be an algebra morphism such that $\alpha^{\otimes 2}(r) = r$ (e.g., $\alpha(b)=b$).  By Corollary \ref{cor:qt1} there is a quasi-triangular $\epsilon$-Hom-bialgebra
\[
A_\alpha = (A,\mualpha=\alpha\circ\mu,\Deltaalpha= \Delta\circ\alpha,\alpha,r).
\]
More explicitly, we have
\[
\Deltaalpha(a) = \alpha(a)b \otimes b - b \otimes b\alpha(a)
\]
for $b \in A$.
\qed
\end{example}

We end this section with several alternative characterizations of the AHYBE in a coboundary $\epsilon$-Hom-bialgebra.  We use the following notations.  Let $(A,\mu,\Delta,\alpha,r)$ be a coboundary $\epsilon$-Hom-bialgebra with $r = \sum u_i \otimes v_i$.  Recall that $A^*$ denotes the linear dual $\Hom(A,\bk)$ of $A$, and $\langle \phi,a\rangle = \phi(a)$ for $\phi \in A^*$ and $a \in A$.  If $A$ is finite dimensional, then $(A^*,\Delta^*,\mu^*,\alpha^*)$ is an $\epsilon$-Hom-bialgebra (see Proposition \ref{prop:dual}).  Actually $\Delta^*$ is defined even if $A$ is not finite dimensional, but $\mu^*$ is defined only when $A$ is finite dimensional.  Define the maps $\lambda_1,\lambda_2,\rho_1,\rho_2 \colon A^* \to A$ as follows:
\begin{equation}
\label{lambdarho}
\begin{split}
\lambda_1(\phi) &= \langle \phi,\alpha(u_i)\rangle v_i,\\
\lambda_2(\phi) &= \langle \phi,u_i\rangle \alpha(v_i),\\
\rho_1(\phi) &= u_i\langle \phi,\alpha(v_i)\rangle,\\
\rho_2(\phi) &= \alpha(u_i) \langle \phi,v_i\rangle
\end{split}
\end{equation}
for $\phi \in A^*$.
The following result gives several characterizations of the AHYBE in $A$.  It is a generalization of (\cite{aguiar} Propositions 5.5 and 5.6).

\begin{theorem}
\label{thm:cobahybe}
Let $(A,\mu,\Delta,\alpha,r)$ be a coboundary $\epsilon$-Hom-bialgebra with $r = \sum u_i \otimes v_i$.  Then the following statements, in which the last two only apply when $A$ is finite dimensional, are equivalent.
\begin{enumerate}
\item
$A$ is a quasi-triangular $\epsilon$-Hom-bialgebra, i.e., $A(r) = 0$ \eqref{ahybe}.
\item
$(\alpha \otimes \Delta)(r) = r_{13}r_{12}$.
\item
$(\Delta \otimes \alpha)(r) = -r_{23}r_{13}$.
\item
The square
\[
\nicearrow
\xymatrix{
A^* \otimes A^* \ar[rr]^{\lambda_1^{\otimes 2}} \ar[d]_-{\Delta^*} & & A \otimes A \ar[d]^-{-\muop}\\
A^* \ar[rr]^-{\lambda_2} & & A
}
\]
is commutative, where $\muop = \mu \circ \tau$.
\item
The square
\[
\nicearrow
\xymatrix{
A^* \otimes A^* \ar[rr]^{\rho_1^{\otimes 2}} \ar[d]_-{\Delta^*} & & A \otimes A \ar[d]^-{\muop}\\
A^* \ar[rr]^-{\rho_2} & & A
}
\]
is commutative.
\item
The square
\[
\nicearrow
\xymatrix{
A^* \ar[rr]^-{\lambda_1} \ar[d]_-{\mustarop} & & A \ar[d]^-{\Delta}\\
A^* \otimes A^* \ar[rr]^-{\lambda_2^{\otimes 2}} & & A \otimes A
}
\]
is commutative, where $\mustarop = \tau \circ \mu^*$.
\item
The square
\[
\nicearrow
\xymatrix{
A^* \ar[rr]^-{\rho_1} \ar[d]_-{\mustarop} & & A \ar[d]^-{-\Delta}\\
A^* \otimes A^* \ar[rr]^-{\rho_2^{\otimes 2}} & & A \otimes A
}
\]
is commutative.
\end{enumerate}
\end{theorem}

\begin{proof}
In the coboundary $\epsilon$-Hom-bialgebra $A$, the comultiplication $\Delta = [-,r]_*$ is given by
\[
\Delta(a) = au_i \otimes \alpha(v_i) - \alpha(u_i) \otimes v_ia,
\]
as discussed in \eqref{starbracketa}.  Thus, with the notations in \eqref{ar}, we have
\begin{equation}
\label{alphadeltar}
\begin{split}
(\alpha \otimes \Delta)(r) &= \alpha(u_i) \otimes \Delta(v_i)\\
&= \alpha(u_i) \otimes v_iu_j \otimes \alpha(v_j) - \alpha(u_i) \otimes \alpha(u_j) \otimes v_jv_i\\
&= r_{12}r_{23} - r_{23}r_{13}.
\end{split}
\end{equation}
This shows the equivalence between statements (1) and (2).  Likewise, we have
\begin{equation}
\label{deltaalphar}
\begin{split}
(\Delta \otimes \alpha)(r) &= \Delta(u_i) \otimes \alpha(v_i)\\
&= u_iu_j \otimes \alpha(v_j) \otimes  \alpha(v_i) - \alpha(u_j) \otimes v_ju_i \otimes \alpha(v_i)\\
&= r_{13}r_{12} - r_{12}r_{23},
\end{split}
\end{equation}
from which the equivalence between statements (1) and (3) follows.

For the equivalence between statements (1) and (4), note that $A(r) = 0$ (as in \eqref{ahybe}) if and only if
\begin{equation}
\label{4'}
\langle \phi \otimes \psi \otimes Id_A, A(r)\rangle = 0
\end{equation}
for all $\phi, \psi \in A^*$.  By \eqref{deltaalphar} we have
\begin{equation}
\label{ar'}
\begin{split}
A(r) &= r_{13}r_{12} - r_{12}r_{23} + r_{23}r_{13}\\
&= (\Delta \otimes \alpha)(r) + r_{23}r_{13}\\
&= \Delta(u_i) \otimes \alpha(v_i) + \alpha(u_i) \otimes \alpha(u_j) \otimes v_jv_i.
\end{split}
\end{equation}
Therefore, \eqref{4'} is equivalent to
\[
\begin{split}
0 &= \langle \phi \otimes \psi \otimes Id_A, \Delta(u_i) \otimes \alpha(v_i) + \alpha(u_i) \otimes \alpha(u_j) \otimes v_jv_i \rangle\\
&= \langle \Delta^*(\phi \otimes \psi),u_i\rangle \alpha(v_i) + \langle \phi,\alpha(u_i)\rangle \langle \psi,\alpha(u_j)\rangle v_jv_i\\
&= \lambda_2(\Delta^*(\phi \otimes \psi)) + \muop(\lambda_1^{\otimes 2}(\phi \otimes \psi)).
\end{split}
\]
This proves the equivalence between statements (1) and (4).

For the equivalence between statements (1) and (5), note that $A(r) = 0$ (as in \eqref{ahybe}) if and only if
\begin{equation}
\label{5'}
\langle Id_A \otimes \phi \otimes \psi, A(r)\rangle = 0
\end{equation}
for all $\phi, \psi \in A^*$.  By \eqref{alphadeltar} we have
\begin{equation}
\label{ar''}
\begin{split}
A(r) &= r_{13}r_{12} - r_{12}r_{23} + r_{23}r_{13}\\
&= r_{13}r_{12} - (\alpha \otimes \Delta)(r)\\
&= u_iu_j \otimes \alpha(v_j) \otimes \alpha(v_i) - \alpha(u_i) \otimes \Delta(v_i).
\end{split}
\end{equation}
Therefore, \eqref{5'} is equivalent to
\[
\begin{split}
0 &= \langle Id_A \otimes \phi \otimes \psi, u_iu_j \otimes \alpha(v_j) \otimes \alpha(v_i) - \alpha(u_i) \otimes \Delta(v_i)\rangle\\
&= u_iu_j\langle \phi,\alpha(v_j)\rangle\langle \psi,\alpha(v_i)\rangle - \alpha(u_i) \langle\Delta^*(\phi \otimes \psi),v_i\rangle\\
&= \muop(\rho_1^{\otimes 2}(\phi \otimes \psi)) - \rho_2(\Delta^*(\psi \otimes \psi)).
\end{split}
\]
This proves the equivalence between statements (1) and (5).

To prove the equivalence between statements (1) and (6), observe that $A(r) = 0$ if and only if
\[
\langle \phi \otimes Id_A \otimes Id_A, A(r)\rangle = 0
\]
for all $\phi \in A^*$.  Then we use \eqref{ar''} and proceed as above.   To prove the equivalence between statements (1) and (7), observe that $A(r) = 0$ if and only if
\[
\langle Id_A \otimes Id_A \otimes \phi, A(r)\rangle = 0
\]
for all $\phi \in A^*$.  Then we use \eqref{ar'} and proceed as above.
\end{proof}

\section{From $\epsilon$-Hom-bialgebras to Hom-Lie bialgebras}
\label{sec:homlie}

The relationship between $\epsilon$-bialgebras and Lie bialgebras has been studied by Aguiar \cite{aguiar2}.  The purpose of this section is to describe a condition on an $\epsilon$-Hom-bialgebra (or a coboundary $\epsilon$-Hom-bialgebra) so that its commutator bracket and cocommutator cobracket give a (coboundary) Hom-Lie bialgebra \cite{yau8} (see Corollary \ref{cor:homliebi} and Theorem \ref{thm:coboundhomlie}).  The transitions from solutions of the AHYBE \eqref{ahybe} to solutions of the CHYBE and from quasi-triangular $\epsilon$-Hom-bialgebras to quasi-triangular Hom-Lie bialgebras are studied in the next section.  The results in this section are Hom-type generalizations of some of those in \cite{aguiar2} (section 4).

As discussed in \cite{yau8} (section 3), a Hom-Lie bialgebra is simultaneously a Hom-Lie algebra and a Hom-Lie coalgebra in which the cobracket is a $1$-cocycle in Hom-Lie algebra cohomology \cite{ms3}.  This suggests a close relationship between $\epsilon$-Hom-bialgebras and Hom-Lie bialgebras, since the comultiplication in an $\epsilon$-Hom-bialgebra is a $1$-cocycle in Hom-associative algebra cohomology (Remark \ref{rk:cocycle}).  This is indeed the case, as we describe below.

Let us first recall the definition of a Hom-Lie bialgebra from \cite{yau8}.

\begin{definition}
\label{def:homlie}
Let $\sigma$ denote the cyclic permutation $(1\, 2\, 3)$.
\begin{enumerate}
\item
A \textbf{Hom-Lie algebra} \cite{hls,ms} $(L,[-,-],\alpha)$ consists of a vector space $L$, a linear map $\alpha \colon L \to L$ (the twisting map), and an anti-symmetric bilinear operation $[-,-] \colon L^{\otimes 2} \to L$ (the bracket) such that the following conditions are satisfied: (1) $\alpha \circ [-,-] = [-,-] \circ \alpha^{\otimes 2}$ (multiplicativity) and (2) the \textbf{Hom-Jacobi identity},
\begin{equation}
\label{homjacobi}
[-,-]\circ([-,-] \otimes \alpha)\circ(Id + \sigma + \sigma^2) = 0.
\end{equation}
For an element $x$ in a Hom-Lie algebra $(L,[-,-],\alpha)$ and $n \geq 2$, define the \emph{adjoint map} $\ad_x \colon L^{\otimes n} \to L^{\otimes n}$ by
\begin{equation}
\label{eq:ad}
\ad_x(y_1 \otimes \cdots \otimes y_n) = \sum_{i=1}^n \alpha(y_1) \otimes \cdots \otimes \alpha(y_{i-1}) \otimes [x,y_i] \otimes \alpha(y_{i+1}) \cdots \otimes \alpha(y_n).
\end{equation}
Conversely, given $\gamma = \sum y_1 \otimes \cdots \otimes y_n$, we define the map $\ad(\gamma) \colon L \to L^{\otimes n}$ by $\ad(\gamma)(x) = \ad_x(\gamma)$.
\item
A \textbf{Hom-Lie coalgebra} \cite{ms2,ms4} $(L,\delta,\alpha)$ consists of a vector space $L$, a linear self $\alpha \colon L \to L$ (the twisting map), and a linear map $\delta \colon L \to L^{\otimes 2}$ (the cobracket) such that the following conditions are satisfied: (1) $\delta \circ \alpha = \alpha^{\otimes 2} \circ \delta$ (comultiplicativity), (2) $\tau \circ \delta = -\delta$ (anti-symmetry), and (3) the \textbf{Hom-co-Jacobi identity}
\begin{equation}
\label{homcojacobi}
(Id + \sigma + \sigma^2) \circ (\alpha \otimes \delta) \circ \delta = 0.
\end{equation}
\item
A \textbf{Hom-Lie bialgebra} $(L,[-,-],\delta,\alpha)$ is a quadruple in which $(L,[-,-],\alpha)$ is a Hom-Lie algebra and $(L,\delta,\alpha)$ is a Hom-Lie coalgebra such that the following compatibility condition holds for all $x, y \in L$:
\begin{equation}
\label{eq:compatibility}
\delta([x,y]) = \ad_{\alpha(x)}(\delta(y)) - \ad_{\alpha(y)}(\delta(x)).
\end{equation}
\end{enumerate}
\end{definition}

In terms of elements $x,y,z \in L$, the Hom-Jacobi identity \eqref{homjacobi} says that
\[
[[x,y],\alpha(z)] + [[z,x],\alpha(y)] + [[y,z],\alpha(x)] = 0.
\]
A Hom-Lie algebra with $\alpha = Id$ is exactly a Lie algebra.  If $(A,\mu,\alpha)$ is a Hom-associative algebra (Definition \ref{def:homas}), then $(A,[-,-],\alpha)$ is a Hom-Lie algebra, where $[-,-] = \mu \circ (Id - \tau)$ is the commutator bracket of $\mu$ \cite{ms} (Proposition 1.7).  Dually, if $(A,\Delta,\alpha)$ is a Hom-coassociative coalgebra, then $(A,\delta,\alpha)$ is a Hom-Lie coalgebra, where $\delta = (Id - \tau) \circ \Delta$ is the cocommutator cobracket of $\Delta$ \cite{ms2} (Proposition 2.8).

A Hom-Lie bialgebra with $\alpha = Id$ is exactly a Lie bialgebra in the sense of Drinfel'd \cite{dri83,dri87}.  From the definition of the adjoint map \eqref{eq:ad}, the compatibility condition \eqref{eq:compatibility} is equivalent to
\begin{equation}
\label{eq:compat}
\begin{split}
\delta([x,y]) &= [\alpha(x),y_1] \otimes \alpha(y_2) + \alpha(y_1) \otimes [\alpha(x),y_2]\\
&\relphantom{}- [\alpha(y),x_1] \otimes \alpha(x_2) - \alpha(x_1) \otimes [\alpha(y),x_2],
\end{split}
\end{equation}
where $\delta(x) = \sum x_1 \otimes x_2$ and $\delta(y) = \sum y_1 \otimes y_2$.  As explained in \cite{yau8}, the condition \eqref{eq:compatibility} says that $\delta$ is a $1$-cocycle in the Hom-Lie algebra cohomology of $L$ with coefficients in $L^{\otimes 2}$ \cite{ms3}.  This is the Hom-type analog of the fact that in a Lie bialgebra, the cobracket is a $1$-cocycle in Chevalley-Eilenberg cohomology with coefficients in the tensor-square.  Properties and examples of Hom-Lie bialgebras and of the corresponding classical Hom-Yang-Baxter equation (CHYBE) can be found in \cite{yau8}.

To state the main result of this section, we need the following concept, which is the Hom-type generalization of the corresponding concept in \cite{aguiar2}.

\begin{definition}
\label{def:balance}
Let $(A,\mu,\Delta,\alpha)$ be an $\epsilon$-Hom-bialgebra.  Define the map $B \colon A^{\otimes 2} \to A^{\otimes 2}$, called the \textbf{balanceator}, by
\begin{equation}
\label{balanceator}
B(x,y) = [x,\Deltaop(y)]_\bullet + \tau([y,\Deltaop(x)]_\bullet)
\end{equation}
for $x,y \in A$, where $\Deltaop = \tau \circ \Delta$ and $[-,-]_\bullet$ is defined in \eqref{dotbracket}.  The balanceator is said to be \textbf{symmetric} if $B = B \circ \tau$, i.e., $B(x,y) = B(y,x)$ for all $x,y \in A$.
\end{definition}

Writing $\Delta(x) = \sum x_1 \otimes x_2$ and $\Delta(y) = \sum y_1 \otimes y_2$, the balanceator is given by
\[
\label{B}
\begin{split}
B(x,y) &= [x,y_2 \otimes y_1]_\bullet + \tau([y,x_2 \otimes x_1]_\bullet)\\
&= \alpha(x)y_2 \otimes \alpha(y_1) - \alpha(y_2) \otimes y_1\alpha(x) + \alpha(x_1) \otimes \alpha(y)x_2 - x_1\alpha(y) \otimes \alpha(x_2).
\end{split}
\]
When $A$ is an $\epsilon$-bialgebra (i.e., $\alpha = Id$), the above definition of the balanceator coincides with that of Aguiar \cite{aguiar2}.

We are now ready to state the main result of this section, which tells us exactly how far \eqref{eq:compatibility} is from being true in an $\epsilon$-Hom-bialgebra in terms of the balanceator.

\begin{theorem}
\label{thm:homliebi}
Let $(A,\mu,\Delta,\alpha)$ be an $\epsilon$-Hom-bialgebra.  Define the commutative bracket $[-,-] = \mu \circ (Id - \tau)$ and the cocommutator cobracket $\delta = (Id - \tau) \circ \Delta$.  Then we have
\begin{equation}
\label{deltaB}
\delta([x,y]) = \ad_{\alpha(x)}(\delta(y)) - \ad_{\alpha(y)}(\delta(x)) + B(x,y) - B(y,x)
\end{equation}
for all $x,y \in A$.
\end{theorem}

For the proof of Theorem \ref{thm:homliebi}, we adapt the shorthand \eqref{adminus} as follows.  If $f(x,y)$ is an expression involving $x$ and $y$, then
\[
(f(x,y))^- = f(x,y) - f(y,x).
\]
In particular, the right-hand side of \eqref{deltaB} becomes $(\ad_{\alpha(x)}(\delta(y)) + B(x,y))^-$.

\begin{proof}[Proof of Theorem \ref{thm:homliebi}]
Using the cocycle condition \eqref{cocycle} in $A$, the left-hand side of \eqref{deltaB} expands as:
\begin{equation}
\label{deltalhs}
\begin{split}
\delta([x,y])
&= (Id - \tau)(\Delta(xy)) - (Id - \tau)(\Delta(yx))\\
&= \alpha(x)y_1 \otimes \alpha(y_2) + \alpha(x_1) \otimes x_2\alpha(y) - \alpha(y_2)\otimes \alpha(x)y_1 - x_2\alpha(y) \otimes \alpha(x_1)\\
&\relphantom{} - \alpha(y)x_1 \otimes \alpha(x_2) - \alpha(y_1) \otimes y_2\alpha(x) + \alpha(x_2) \otimes \alpha(y)x_1 + y_2\alpha(x) \otimes \alpha(y_1)\\
&= \left(\alpha(x)y_1 \otimes \alpha(y_2 )- \alpha(y_1) \otimes y_2\alpha(x) + y_2\alpha(x) \otimes \alpha(y_1) - \alpha(y_2) \otimes \alpha(x)y_1\right)^-.
\end{split}
\end{equation}
For the right-hand side of \eqref{deltaB}, first observe that:
\begin{equation}
\label{deltarhs1}
\begin{split}
\ad_{\alpha(x)}(\delta(y))
&= \ad_{\alpha(x)}(y_1 \otimes y_2 - y_2 \otimes y_1)\\
&= [\alpha(x),y_1] \otimes \alpha(y_2) + \alpha(y_1) \otimes [\alpha(x),y_2] - [\alpha(x),y_2] \otimes \alpha(y_1) - \alpha(y_2) \otimes [\alpha(x),y_1]\\
&= \alpha(x)y_1 \otimes \alpha(y_2) - y_1\alpha(x) \otimes \alpha(y_2) + \alpha(y_1) \otimes \alpha(x)y_2 - \alpha(y_1) \otimes y_2\alpha(x)\\
&\relphantom{} - \alpha(x)y_2 \otimes \alpha(y_1) + y_2\alpha(x) \otimes \alpha(y_1) - \alpha(y_2) \otimes \alpha(x)y_1 + \alpha(y_2) \otimes y_1\alpha(x)\\
&= \alpha(x)y_1 \otimes \alpha(y_2 )- \alpha(y_1) \otimes y_2\alpha(x) + y_2\alpha(x) \otimes \alpha(y_1) - \alpha(y_2) \otimes \alpha(x)y_1\\
&\relphantom{} + \tau([x,\Deltaop(y)]_\bullet) - [x,\Deltaop(y)]_\bullet.
\end{split}
\end{equation}
Using \eqref{deltalhs} and \eqref{deltarhs1}, the right-hand side of \eqref{deltaB} is:
\[
\begin{split}
&(\ad_{\alpha(x)}(\delta(y)) + B(x,y))^-\\
&= (\alpha(x)y_1 \otimes \alpha(y_2 )- \alpha(y_1) \otimes y_2\alpha(x) + y_2\alpha(x) \otimes \alpha(y_1) - \alpha(y_2) \otimes \alpha(x)y_1)^-\\
&\relphantom{} + (\tau([x,\Deltaop(y)]_\bullet) + \tau([y,\Deltaop(x)]_\bullet))^-\\
&= \delta([x,y]),
\end{split}
\]
as desired.
\end{proof}

In the context of Theorem \ref{thm:homliebi}, $(A,[-,-],\alpha)$ is a Hom-Lie algebra, and $(A,\delta,\alpha)$ is a Hom-Lie coalgebra.  Comparing the compatibility condition \eqref{eq:compatibility} in a Hom-Lie bialgebra with \eqref{deltaB}, we obtain the following immediate consequence of Theorem \ref{thm:homliebi}.  It tells us exactly when one can associate a Hom-Lie bialgebra to an $\epsilon$-Hom-bialgebra via the (co)commutator (co)bracket.

\begin{corollary}
\label{cor:homliebi}
Let $(A,\mu,\Delta,\alpha)$ be an $\epsilon$-Hom-bialgebra.  Then
\[
A_{HLie} = (A,[-,-],\delta,\alpha)
\]
is a Hom-Lie bialgebra if and only if the balanceator $B$ is symmetric, where $[-,-] = \mu \circ (Id - \tau)$ and $\delta = (Id - \tau) \circ \Delta$.
\end{corollary}

Next we discuss construction results for $\epsilon$-Hom-bialgebras whose balanceators are symmetric, for which Corollary \ref{cor:homliebi} can be used to construct Hom-Lie bialgebras.  The following result is the First Twisting Principle for such $\epsilon$-Hom-bialgebras.

\begin{theorem}
\label{thm:Btp}
Let $(A,\mu,\Delta)$ be an $\epsilon$-bialgebra with balanceator $B$ and $\alpha \colon A \to A$ be an $\epsilon$-bialgebra morphism.  Then the balanceator $B_\alpha$ of the $\epsilon$-Hom-bialgebra $A_\alpha$ is given by
\[
B_\alpha = (\alpha^2)^{\otimes 2} \circ B,
\]
where $A_\alpha = (A,\mualpha = \alpha\circ\mu,\Deltaalpha=\Delta\circ\alpha,\alpha)$ is the $\epsilon$-Hom-bialgebra in Theorem \ref{thm:firsttp}.  In particular, if $B$ is symmetric, then so is $B_\alpha$.  Conversely, if $B_\alpha$ is symmetric and $\alpha^{\otimes 2}$ is injective, then $B$ is symmetric.
\end{theorem}

\begin{proof}
For an element $y \in A$, write $\Delta(y) = \sum y_1 \otimes y_2$.  Then we have $\Deltaalpha(y) = \sum \alpha(y_1) \otimes \alpha(y_2)$.  In the $\epsilon$-Hom-bialgebra $A_\alpha$, we have
\[
\begin{split}
[x,\Deltaalphaop(y)]_\bullet
&= x \bullet (\alpha(y_2) \otimes \alpha(y_1)) - (\alpha(y_2) \otimes \alpha(y_1)) \bullet x\\
&= \mualpha(\alpha(x),\alpha(y_2)) \otimes \alpha^2(y_1) - \alpha^2(y_2) \otimes \mualpha(\alpha(y_1),\alpha(x))\\
&= (\alpha^2)^{\otimes 2}(\mu(x,y_2) \otimes y_1 - y_2 \otimes \mu(y_1,x))\\
&= (\alpha^2)^{\otimes 2}([x,\Deltaop(y)]).
\end{split}
\]
It follows that
\[
\begin{split}
B_\alpha(x,y)
&= [x,\Deltaalphaop(y)]_\bullet + \tau[y,\Deltaalphaop(x)]_\bullet\\
&= (\alpha^2)^{\otimes 2}([x,\Deltaop(y)] + \tau[y,\Deltaop(x)])\\
&= (\alpha^2)^{\otimes 2}(B(x,y)),
\end{split}
\]
as desired.  The last two assertions of the Theorem follow immediately from the first assertion, which we just established.
\end{proof}


\begin{example}
Let $(A,\mu,\Delta)$ be an $\epsilon$-bialgebra that is both commutative and cocommutative.  Then its balanceator $B = 0$ (\cite{aguiar2} Proposition 4.5). If $\alpha$ is any $\epsilon$-bialgebra morphism on $A$, then it follows from Theorem \ref{thm:Btp} that $B_\alpha = 0$, which is trivially symmetric.  In this case, Corollary \ref{cor:homliebi} applies to the $\epsilon$-Hom-bialgebra $A_\alpha$ to yield a Hom-Lie bialgebra.\qed
\end{example}

\begin{example}
Let $(A,\mu,\Delta,r)$ be a coboundary $\epsilon$-bialgebra in which $r^+$ is $A$-invariant, where $r^+ = (r + \tau(r))/2$.  Then its balanceator $B = 0$ (\cite{aguiar2} Proposition 4.7).  If $\alpha$ is any $\epsilon$-bialgebra morphism on $A$, then it follows from Theorem \ref{thm:Btp} that $B_\alpha = 0$, which is trivially symmetric.  In this case, Corollary \ref{cor:homliebi} applies to the $\epsilon$-Hom-bialgebra $A_\alpha$ to yield a Hom-Lie bialgebra.\qed
\end{example}

\begin{example}
Let $A$ be the truncated polynomial algebra $\bk[x]/(x^4)$.  It is an $\epsilon$-bialgebra with comultiplication
\[
\Delta(x^i) = x^i \otimes x^2 - 1 \otimes x^{i+2},
\]
i.e., $\Delta(1) = 0$, $\Delta(x) = x \otimes x^2 - 1 \otimes x^3$, $\Delta(x^2) = x^2 \otimes x^2$, and $\Delta(x^3) = x^3 \otimes x^2$.  Its balanceator $B$ is symmetric (but not $= 0$).  This is Example 4.12.2 in \cite{aguiar2}.  If $\alpha \colon A \to A$ is any $\epsilon$-bialgebra morphism, then it follows from Theorem \ref{thm:Btp} that $B_\alpha$ is symmetric.  In this case, Corollary \ref{cor:homliebi} applies to the $\epsilon$-Hom-bialgebra $A_\alpha$ to yield a Hom-Lie bialgebra.

It is not hard to classify the unit-preserving $\epsilon$-bialgebra morphisms on $A$.   In fact, a unit-preserving algebra morphism $\alpha \colon A \to A$ is determined by $\alpha(x) = a_0 + a_1x + a_2x^2 + a_3x^3$.  By Proposition \ref{prop:morphism} such an $\alpha$ is an $\epsilon$-bialgebra morphism if and only if $\alpha^{\otimes 2}(\Delta(x)) = \Delta(\alpha(x))$.  One can solve this equation for the $a_i$ with an elementary but tedious computation. There are exactly two classes of solutions: Either $\alpha(x) = a_0$ or $\alpha(x) = \pm x + a_3x^3$ with $a_0, a_3 \in \bk$.
\qed
\end{example}

The next result is the Second Twisting Principle for $\epsilon$-Hom-bialgebras whose balanceators are symmetric, for which Corollary \ref{cor:homliebi} can be used to construct Hom-Lie bialgebras.

\begin{theorem}
\label{thm:Btp2}
Let $(A,\mu,\Delta,\alpha)$ be an $\epsilon$-Hom-bialgebra with balanceator $B$.  Then for each $n \geq 1$, the balanceator $B^n$ of $A^n$ is given by
\[
B^n = (\alpha^{2(2^n-1)})^{\otimes 2} \circ B,
\]
where $A^n = (A,\mun,\Deltan,\alpha^{2^n})$ is the $n$th derived $\epsilon$-Hom-bialgebra of $A$ with $\mun = \alpha^{2^n-1}\circ\mu$ and $\Deltan = \Delta \circ \alpha^{2^n-1}$ (Theorem \ref{thm:secondtp}).  In particular, if $B$ is symmetric, then so is $B^n$.  Conversely, if $B^n$ is symmetric for some $n$ and $\alpha^{\otimes 2}$ is injective, then $B$ is symmetric.
\end{theorem}

\begin{proof}
Consider the case $n=1$.  For $y \in A$, write $\Delta(y) = \sum y_1 \otimes y_2$.  In $A^1$ we have:
\[
\begin{split}
[x,\Deltaoneop(y)]_\bullet
&= x \bullet (\alpha(y_2) \otimes \alpha(y_1)) - (\alpha(y_2) \otimes \alpha(y_1)) \bullet x\\
&= \muone(\alpha^2(x),\alpha(y_2)) \otimes \alpha^3(y_1) - \alpha^3(y_2) \otimes \muone(\alpha(y_1),\alpha^2(x))\\
&= (\alpha^2)^{\otimes 2}(\mu(\alpha(x),y_2) \otimes \alpha(y_1) - \alpha(y_2) \otimes \mu(y_1,\alpha(x))).
\end{split}
\]
Therefore, in $A^1$ we have
\begin{equation}
\label{B1}
\begin{split}
B^1(x,y) &= [x,\Deltaoneop(y)]_\bullet + \tau[y,\Deltaoneop(x)]_\bullet\\
&= (\alpha^2)^{\otimes 2}(\mu(\alpha(x),y_2) \otimes \alpha(y_1) - \alpha(y_2) \otimes \mu(y_1,\alpha(x)))\\
&\relphantom{} + (\alpha^2)^{\otimes 2}\circ\tau(\mu(\alpha(y),x_2) \otimes \alpha(x_1) - \alpha(x_2) \otimes \mu(x_1,\alpha(y)))\\
&= (\alpha^2)^{\otimes 2}(B(x,y)),
\end{split}
\end{equation}
as desired.  For the general case, we proceed inductively, noting that $A^n = (A^{n-1})^1$ and using the case just proved repeatedly.  We have:
\[
\begin{split}
B^n &= ((\alpha^{2^{n-1}})^2)^{\otimes 2} \circ B^{n-1}\quad\text{(by \eqref{B1} applied to $A^{n-1}$)}\\
&= (\alpha^{2^n})^{\otimes 2} \circ (\alpha^{2^n - 2})^{\otimes 2} \circ B \quad \text{(by induction hypothesis)}\\
&= (\alpha^{2^{n+1} - 2})^{\otimes 2} \circ B,
\end{split}
\]
as was to be shown.  The last two assertions follow immediately from the first one, which we just established.
\end{proof}

Combining Corollary \ref{cor:homliebi} and Theorem \ref{thm:Btp2}, we obtain a sequence of Hom-Lie bialgebras from any $\epsilon$-Hom-bialgebra with symmetric balanceator.

\begin{corollary}
\label{cor:anhomliebi}
Let $(A,\mu,\Delta,\alpha)$ be an $\epsilon$-Hom-bialgebra whose balanceator $B$ is symmetric.  Then
\[
(A^n)_{HLie} = (A,\mun\circ(Id - \tau), (Id-\tau)\circ\Deltan,\alpha^{2^n})
\]
is a Hom-Lie bialgebra for each $n \geq 1$, where $\mun = \alpha^{2^n-1}\circ\mu$ and $\Deltan = \Delta \circ \alpha^{2^n-1}$.
\end{corollary}


Now we restrict to the subclass of coboundary $\epsilon$-Hom-bialgebra (Definition \ref{def:cobound}).  The corresponding Hom-Lie objects are called coboundary Hom-Lie bialgebras.

\begin{definition}
\label{def:coboundhomlie}
A \textbf{coboundary Hom-Lie bialgebra} $(L,[-,-],\delta,\alpha,r)$ consists of a Hom-Lie bialgebra $(L,[-,-],\delta,\alpha)$ (Definition \ref{def:homlie}) and an $\alpha$-invariant $r \in L^{\otimes 2}$ such that $\delta = \ad(r)$ \eqref{eq:ad}.
\end{definition}

Coboundary Hom-Lie bialgebras were introduced and studied in \cite{yau8}.  They are the Hom-type analogs of Drinfel'd's coboundary Lie bialgebras \cite{dri83,dri87}.  As explained in \cite{yau8}, the operators $\ad(r)$ are, in fact, $1$-coboundaries in Hom-Lie algebra cohomology \cite[section 5]{ms3} with coefficients in the tensor-square.  This justifies the terminology.

A $2$-tensor $r \in A^{\otimes 2}$ is said to be \textbf{anti-symmetric} if $r^{op} = -r$, where $r^{op} = \tau(r)$.  The following result gives a sufficient condition under which a coboundary $\epsilon$-Hom-bialgebra gives rise to a coboundary Hom-Lie bialgebra via the (co)commutator (co)bracket.

\begin{theorem}
\label{thm:coboundhomlie}
Let $(A,\mu,\Delta=[-,r]_*,\alpha,r)$ be a coboundary $\epsilon$-Hom-bialgebra with $r$ anti-symmetric.  Then
\[
A_{HLie} = (A,[-,-]=\mu\circ(Id-\tau),\delta=(Id-\tau)\circ\Delta,\alpha,r)
\]
is a coboundary Hom-Lie bialgebra.
\end{theorem}

\begin{proof}
In Lemma \ref{lem1:coboundhomlie} below, we will show that the balanceator $B$ of $A$ is $0$, which is trivially symmetric.  Thus Corollary \ref{cor:homliebi} implies that $A_{HLie}$ is a Hom-Lie bialgebra.  It remains to show that $\delta = \ad(r)$.  Writing $r = \sum u_i \otimes v_i$, for $a \in A$ we compute as follows:
\[
\begin{split}
\delta(a) &= (Id - \tau)(\Delta(a))\\
&= (Id - \tau)(au_i \otimes \alpha(v_i) - \alpha(u_i) \otimes v_ia)\quad \text{(by \eqref{starbracketa})}\\
&= au_i \otimes \alpha(v_i) - \alpha(u_i) \otimes v_ia - \alpha(v_i) \otimes au_i + v_ia \otimes \alpha(u_i)\\
&= au_i \otimes \alpha(v_i) - u_ia \otimes \alpha(v_i) + \alpha(u_i) \otimes av_i - \alpha(u_i) \otimes v_ia\\
&= [a,u_i] \otimes \alpha(v_i) + \alpha(u_i) \otimes [a,v_i]\\
&= \ad_a(r).
\end{split}
\]
In the fourth equality above, we used the anti-symmetry of $r$ twice.  To finish the proof of the Theorem, it remains to prove the following Lemma.
\end{proof}

\begin{lemma}
\label{lem1:coboundhomlie}
With the hypotheses of Theorem \ref{thm:coboundhomlie}, the balanceator $B$ of $A$ is $0$.
\end{lemma}

\begin{proof}
Using the axioms \eqref{homassaxioms} in a Hom-associative algebra and $r^{op} = -r$, for $x,y \in A$ we have:
\[
\begin{split}
B(x,y) &= [x,\Deltaop(y)]_\bullet + \tau[y,\Deltaop(x)]_\bullet\\
&= [x,\alpha(v_i) \otimes yu_i - v_iy \otimes \alpha(u_i)]_\bullet + \tau[y,\alpha(v_i) \otimes xu_i - v_ix \otimes \alpha(u_i)]_\bullet\\
&= \alpha(x)\alpha(v_i) \otimes \alpha(yu_i) - \alpha(x)(v_iy) \otimes \alpha^2(u_i)\\
&\relphantom{} - \alpha^2(v_i) \otimes (yu_i)\alpha(x) + \alpha(v_iy) \otimes \alpha(u_i)\alpha(x)\\
&\relphantom{} + \alpha(xu_i) \otimes \alpha(y)\alpha(v_i) - \alpha^2(u_i) \otimes \alpha(y)(v_ix)\\
&\relphantom{} - (xu_i)\alpha(y) \otimes \alpha^2(v_i) + \alpha(u_i)\alpha(y) \otimes \alpha(v_ix)\\
&= -\alpha(xu_i) \otimes \alpha(yv_i) + \alpha(x)(u_iy) \otimes \alpha^2(v_i)\\
&\relphantom{} + \alpha^2(u_i) \otimes (yv_i)\alpha(x) - \alpha(u_iy) \otimes \alpha(v_ix)\\
&\relphantom{} + \alpha(xu_i) \otimes \alpha(yv_i) - \alpha^2(u_i) \otimes (yv_i)\alpha(x)\\
&\relphantom{} - \alpha(x)(u_iy) \otimes \alpha^2(v_i) + \alpha(u_iy) \otimes \alpha(v_ix)\\
&= 0.
\end{split}
\]
This proves the Lemma and hence also Theorem \ref{thm:coboundhomlie}.
\end{proof}

\section{Quasi-triangular $\epsilon$-Hom-bialgebras and Hom-Lie bialgebras}
\label{sec:qt}

In the previous section, we discussed the transition from $\epsilon$-Hom-bialgebras to Hom-Lie bialgebras.  In this section, we study the transition from quasi-triangular $\epsilon$-Hom-bialgebras (Definition \ref{def:ahybe}) to their corresponding Hom-Lie objects.  First we need to discuss the transition from solutions of the AHYBE \eqref{ahybe} to solutions of the CHYBE.  The relationships between quasi-triangular $\epsilon$-bialgebras and Lie bialgebras and between the AYBE and the CYBE were studied by Aguiar \cite{aguiar2}.  The results in this section are Hom-type analogs of some of those in \cite{aguiar2} (section 3).

Let us first recall some relevant definitions from \cite{yau8}.

\begin{definition}
\label{def:chybe}
Let $(L,[-,-],\alpha)$ be a Hom-Lie algebra (Definition \ref{def:homlie}) and $r = \sum u_i \otimes v_i, s = \sum u_j' \otimes v_j' \in L^{\otimes 2}$.  Define the following $3$-tensors:
\begin{equation}
\label{cr}
\begin{split}
[r_{12},s_{13}] &= \sum [u_i,u_j'] \otimes \alpha(v_i) \otimes \alpha(v_j'),\\
[r_{12},s_{23}] &= \sum \alpha(u_i) \otimes [v_i,u_j'] \otimes \alpha(v_j'),\\
[r_{13},s_{23}]&= \sum \alpha(u_i) \otimes \alpha(u_j') \otimes [v_i,v_j'],\\
C(r) &= [r_{12},r_{13}] + [r_{12},r_{23}] + [r_{13},r_{23}].
\end{split}
\end{equation}
We say that $r$ is a solution of the \textbf{classical Hom-Yang-Baxter equation} (CHYBE) in $L$ if
\begin{equation}
\label{chybe}
C(r) = 0.
\end{equation}
A \textbf{quasi-triangular Hom-Lie bialgebra} is a coboundary Hom-Lie bialgebra (Definition \ref{def:coboundhomlie}) $(L,[-,-],\delta,\alpha,r)$ in which $r$ is a solution of the CHYBE.
\end{definition}

In the special case $\alpha = Id$ (i.e., $L$ is a Lie algebra), the CHYBE reduces to the \textbf{classical Yang-Baxter equation} (CYBE) \cite{skl1,skl2}.  A quasi-triangular Hom-Lie bialgebra with $\alpha = Id$ is a \textbf{quasi-triangular Lie bialgebra} in the sense of Drinfel'd \cite{dri83,dri87}.  The CHYBE and quasi-triangular Hom-Lie bialgebras were introduced and studied in \cite{yau8}.

Observe the formal similarity between the AHYBE \eqref{ahybe} and the CHYBE \eqref{chybe}.  There is, in fact, a close relationship between them, as we now describe.  Recall that a $2$-tensor $r \in A^{\otimes 2}$ is anti-symmetric if $r^{op} = -r$, where $r^{op} = \tau(r)$.  It is said to be \textbf{symmetric} if $r^{op} = r$.  Also recall that, if $(A,\mu,\alpha)$ is a Hom-associative algebra, then $A_{HLie} = (A,[-,-],\alpha)$ is a Hom-Lie algebra \cite{ms} (Proposition 1.7), where $[-,-] = \mu\circ(Id - \tau)$.

\begin{theorem}
\label{thm:chybe}
Let $(A,\mu,\alpha)$ be a Hom-associative algebra and $r \in A^{\otimes 2}$ be a solution of the AHYBE \eqref{ahybe}.  Suppose that $r$ is either symmetric or anti-symmetric.  Then $r$ is a solution of the CHYBE \eqref{chybe} in the Hom-Lie algebra $A_{HLie} = (A,[-,-] = \mu\circ(Id - \tau),\alpha)$.
\end{theorem}

To prove Theorem \ref{thm:chybe}, we consider the following variation of the $3$-tensor $A(r)$.  With $r = \sum u_i \otimes v_i \in A^{\otimes 2}$, define the $3$-tensor
\begin{equation}
\label{Ar'}
A(r)' = u_ju_i \otimes \alpha(v_j) \otimes \alpha(v_i) - \alpha(u_i) \otimes u_jv_i \otimes \alpha(v_j) + \alpha(u_i) \otimes \alpha(u_j) \otimes v_iv_j.
\end{equation}
Note that, following the conventions in \eqref{ar}, we can identify $A(r)'$ as $r_{12}r_{13} - r_{23}r_{12} + r_{13}r_{23}$.  The following two Lemmas are needed in the proof of Theorem \ref{thm:chybe}.

\begin{lemma}
\label{lem1:chybe}
Let $(A,\mu,\alpha)$ be a Hom-associative algebra and $r \in A^{\otimes 2}$.  Then
\[
C(r) = A(r)' - A(r),
\]
where $C(r)$ \eqref{cr} is defined in the Hom-Lie algebra $A_{HLie} = (A,[-,-] = \mu\circ(Id - \tau),\alpha)$.
\end{lemma}

\begin{proof}
We compute as follows:
\[
\begin{split}
-C(r) &= -[r_{12},r_{13}] - [r_{12},r_{23}] - [r_{13},r_{23}]\\
&= -[u_j,u_i] \otimes \alpha(v_j) \otimes \alpha(v_i) - \alpha(u_i) \otimes [v_i,u_j] \otimes \alpha(v_j) - \alpha(u_i) \otimes \alpha(u_j) \otimes [v_i,v_j]\\
&= u_iu_j \otimes \alpha(v_j) \otimes \alpha(v_i) - u_ju_i \otimes \alpha(v_j) \otimes \alpha(v_i) - \alpha(u_i) \otimes v_iu_j \otimes \alpha(v_j)\\
&\relphantom{} + \alpha(u_i) \otimes u_jv_i \otimes \alpha(v_j) + \alpha(u_i) \otimes \alpha(u_j) \otimes v_jv_i - \alpha(u_i) \otimes \alpha(u_j) \otimes v_iv_j\\
&= A(r) - A(r)'.
\end{split}
\]
In the last equality above, $A(r)$ consists of the first, the third, and the fifth terms, and $-A(r)'$ consists of the second, the fourth, and the sixth terms.
\end{proof}

\begin{lemma}
\label{lem2:chybe}
Let $(A,\mu,\alpha)$ be a Hom-associative algebra and $r \in A^{\otimes 2}$ that is either symmetric or anti-symmetric.  Then $A(r)' = \pi(A(r))$, where $\pi$ denotes the transposition $(1~3)$ on $A^{\otimes 3}$.
\end{lemma}

\begin{proof}
We compute as follows:
\[
\begin{split}
\pi(A(r)) &= \pi(r_{23}r_{13} - r_{12}r_{23} + r_{13}r_{12})\\
&= v_jv_i \otimes \alpha(u_j) \otimes \alpha(u_i) - \alpha(v_j) \otimes v_iu_j \otimes \alpha(u_i) + \alpha(v_i) \otimes \alpha(v_j) \otimes u_iu_j\\
&= u_ju_i \otimes \alpha(v_j) \otimes \alpha(v_i) - \alpha(u_i) \otimes u_jv_i \otimes \alpha(v_j) + \alpha(u_i) \otimes \alpha(u_j) \otimes v_iv_j\\
&= A(r)'.
\end{split}
\]
In the third equality above, we used the (anti-)symmetry of $r$ in each of its six copies.
\end{proof}

\begin{proof}[Proof of Theorem \ref{thm:chybe}]
We have
\[
\begin{split}
C(r) &= A(r)' - A(r) \quad \text{(by Lemma \ref{lem1:chybe})}\\
&= (\pi - Id)(A(r)) \quad \text{(by Lemma \ref{lem2:chybe})}\\
&= 0,
\end{split}
\]
since by assumption $A(r) = 0$.
\end{proof}


Let us give some examples of (anti-)symmetric solutions of the AHYBE, for which Theorem \ref{thm:chybe} can be applied to yield solutions of the CHYBE.

\begin{example}
Let $(A,\mu)$ be an associative algebra, $\alpha \colon A \to A$ be an algebra morphism, and $r \in A^{\otimes 2}$ be a solution of the AYBE that is (anti-)symmetric.  Then for each $n \geq 0$, $(\alpha^{\otimes 2})^n(r)$ is a solution of the AHYBE in the Hom-associative algebra $A_\alpha = (A,\mualpha = \alpha\circ\mu,\alpha)$ that is (anti-)symmetric.  This is an immediate consequence of Theorem \ref{thm:ahybe1}.\qed
\end{example}

\begin{example}
Let $(A,\mu,\alpha)$ be a Hom-associative algebra and $r \in A^{\otimes 2}$ be a solution of the AHYBE that is (anti-)symmetric.  Then for each $n \geq 0$, $r$ is also a solution of the AHYBE in the Hom-associative algebra $A^n = (A,\mun = \alpha^{2^n-1} \circ \mu,\alpha^{2^n})$ that is (anti-)symmetric.  This is a special case of Theorem \ref{thm:ahybe2}.\qed
\end{example}

\begin{example}
Let $M_2$ denote the associative algebra of $2$-by-$2$ complex matrices, and let $\alpha \colon M_2 \to M_2$ be an algebra morphism.  For example, if $\gamma \in M_2$ is invertible, then $\alpha = \gamma(-)\gamma^{-1}$ (i.e., conjugation by $\gamma$) is an algebra morphism.  By \cite{aguiar2} (Example 2.8), up to conjugation and transposition, the only non-zero anti-symmetric solution of the AYBE in $M_2$ is
\[
r = \eoneone \otimes \eonetwo - \eonetwo \otimes \eoneone.
\]
It follows from Theorem \ref{thm:ahybe1} that, for each $n \geq 0$, $(\alpha^{\otimes 2})^n(r)$ is an anti-symmetric solution of the AHYBE in the Hom-associative algebra $(M_2)_\alpha = (M_2,\mualpha =\alpha\circ\mu,\alpha)$.\qed
\end{example}

\begin{example}
\label{ex:2d}
Consider the two-dimensional associative algebra $A = \bk\langle x\rangle \oplus \bk\langle y \rangle$ whose multiplication is determined by
\begin{equation}
\label{2dmult}
x^2 = 0 = xy,\quad yx = x, \quad\text{and}\quad y^2 = y.
\end{equation}
By \cite{aguiar2} (Example 2.3.1) there is an anti-symmetric solution of the AYBE in $A$ given by
\[
r = y \otimes x - x \otimes y.
\]
If $\alpha \colon A \to A$ is any algebra morphism, then by Theorem \ref{thm:ahybe1} $(\alpha^{\otimes 2})^n(r)$ is an anti-symmetric solution of the AHYBE in the Hom-associative algebra $A_\alpha = (A,\mualpha=\alpha\circ\mu,\alpha)$ for each $n \geq 0$.  Let us classify the non-zero algebra morphisms on $A$ and describe their corresponding anti-symmetric solutions of the AHYBE in $A_\alpha$.

A non-zero algebra morphism $\alpha \colon A \to A$ is determined by
\[
\alpha(x) = ax + by \quad\text{and}\quad \alpha(y) = cx + dy
\]
for some $a,b,c,d \in \bk$.  Solving the simultaneously equations
\[
\alpha(x)^2 = 0 = \alpha(x)\alpha(y),\quad
\alpha(y)\alpha(x) = \alpha(x),\quad\text{and}\quad
\alpha(y)^2 = \alpha(y)
\]
for $a,b,c$, and $d$, we conclude that the only non-zero algebra morphisms on $A$ are given by:
\begin{equation}
\label{alphaxy}
\alpha(x) = ax, \quad\alpha(y) = cx + y
\end{equation}
for arbitrary $a,c \in \bk$.  In the Hom-associative algebra $A_\alpha$, the multiplication $\mualpha = \alpha\circ\mu$ is determined by
\begin{equation}
\label{mualphaxy}
\mualpha(x,x) = 0 = \mualpha(x,y),\quad
\mualpha(y,x) = ax,\quad\text{and}\quad
\mualpha(y,y) = cx + y.
\end{equation}

If $a=0$ in $\alpha$ \eqref{alphaxy}, then $\alpha(x) = 0$ and hence $(\alpha^{\otimes 2})^n(r) = 0$ for all $n \geq 1$.  In this case, $r$ is an anti-symmetric solution of the AHYBE in $A_\alpha$.

Now assume that $a\not= 0$.  To describe $(\alpha^{\otimes 2})^n(r)$, first note that
\[
\alpha^n(x) = a^nx.
\]
We claim that
\[
\alpha^n(y) = c\left(\sum_{j=0}^{n-1} a^j\right)x + y
\]
for all $n \geq 1$, where $a^0 \equiv 1$.  The case $n=1$ is true by \eqref{alphaxy}.  Inductively, we have:
\[
\begin{split}
\alpha^{n+1}(y) &= \alpha(\alpha^n(y))\\
&= c(1 + a + \cdots + a^{n-1})\alpha(x) + \alpha(y) \quad(\text{by induction hypothesis})\\
&= c(1 + a + \cdots + a^{n-1})ax + cx + y\\
&= c(1 + a + \cdots + a^n)x + y.
\end{split}
\]
This proves the claim.  Going back to $(\alpha^{\otimes 2})^n(r)$, for $n \geq 1$ we have:
\[
\begin{split}
(\alpha^{\otimes 2})^n(r)
&= \alpha^n(y) \otimes \alpha^n(x) - \alpha^n(x) \otimes \alpha^n(y)\\
&= (c(1 + a + \cdots + a^{n-1})x + y) \otimes a^nx - a^nx \otimes (c(1 + a + \cdots + a^{n-1})x + y)\\
&= a^n(y \otimes x - x \otimes y)\\
&= a^nr.
\end{split}
\]
Therefore, when $a\not= 0$ in $\alpha$ \eqref{alphaxy}, $(\alpha^{\otimes 2})^n(r) = a^nr$ is an anti-symmetric solution of the AHYBE in $A_\alpha$ for each $n \geq 0$.
\qed
\end{example}


Finally, we consider the transition from quasi-triangular $\epsilon$-Hom-bialgebras (Definition \ref{def:ahybe}) to quasi-triangular Hom-Lie bialgebras (Definition \ref{def:chybe}).

\begin{corollary}
\label{cor:qthomlie}
Let $(A,\mu,\Delta,\alpha,r)$ be a quasi-triangular $\epsilon$-Hom-bialgebra with $r$ anti-symmetric.  Then
\[
A_{HLie} = (A,[-,-]=\mu\circ(Id-\tau),\delta=(Id-\tau)\circ\Delta,\alpha,r)
\]
is a quasi-triangular Hom-Lie bialgebra.
\end{corollary}

\begin{proof}
By Theorem \ref{thm:coboundhomlie} we already know that $A_{HLie}$ is a coboundary Hom-Lie bialgebra.  Moreover, by assumption $r$ is an anti-symmetric solution of the AHYBE in $A$.  It follows from Theorem \ref{thm:chybe} that $r$ is a solution of the CHYBE in $A_{HLie}$.
\end{proof}

\begin{example}
Consider the $2$-dimensional Hom-associative algebra $A_\alpha = (A,\mualpha=\alpha\circ\mu,\alpha)$ in Example \ref{ex:2d}, where $A = \bk\langle x\rangle \oplus \bk\langle y \rangle$ with multiplication \eqref{2dmult}.  In this Example, we take $a = 1$ in $\alpha$ \eqref{alphaxy}, so
\[
\alpha(x) = x, \quad \alpha(y) = cx + y
\]
with $c \in \bk$ arbitrary.  It follows that $r = y \otimes x - x \otimes y$ is $\alpha$-invariant, i.e., $\alpha^{\otimes 2}(r) = r$.  Thus, $r$ is an $\alpha$-invariant anti-symmetric solution of the AHYBE in $A_\alpha$ (Theorem \ref{thm:ahybe1}).  By Corollary \ref{cor:qtchar},
\[
A_\alpha = (A,\mualpha=\alpha\circ\mu,\Delta=[-,r]_*,\alpha,r)
\]
is a quasi-triangular $\epsilon$-Hom-bialgebra with $r$ anti-symmetric.  Therefore, Corollary \ref{cor:qthomlie} can be applied to $A_\alpha$ to yield a quasi-triangular Hom-Lie bialgebra
\[
(A_\alpha)_{HLie} = (A,[-,-]=\alpha\circ\mu\circ(Id-\tau),\delta=(Id-\tau)\circ[-,r]_*,\alpha,r)
\]
via the (co)commutator (co)bracket.

More explicitly, using \eqref{starbracketa} and \eqref{mualphaxy}, the comultiplication $\Delta = [-,r]_*$ in $A_\alpha$ is determined by
\[
\begin{split}
\Delta(x) &= [x,y \otimes x - x \otimes y]_*\\
&= \mualpha(x,y) \otimes \alpha(x) - \mualpha(x,x) \otimes \alpha(y) - \alpha(y) \otimes \mualpha(x,x) + \alpha(x) \otimes \mualpha(y,x)\\
&= x \otimes x
\end{split}
\]
and
\[
\begin{split}
\Delta(y) &= [y,y \otimes x - x \otimes y]_*\\
&= \mualpha(y,y) \otimes \alpha(x) - \mualpha(y,x) \otimes \alpha(y) - \alpha(y) \otimes \mualpha(x,y) + \alpha(x) \otimes \mualpha(y,y)\\
&= (cx + y) \otimes x - x \otimes (cx + y) + x \otimes (cx + y)\\
&= cx \otimes x + y \otimes x
\end{split}
\]
Therefore, the cocommutator cobracket $\delta$ in $(A_\alpha)_{HLie}$ is determined by
\[
\begin{split}
\delta(x) &= (Id - \tau)(\Delta(x))\\
&= (Id - \tau)(x \otimes x)\\
&= 0
\end{split}
\]
and
\[
\begin{split}
\delta(y) &= (Id - \tau)(\Delta(y))\\
&= (Id - \tau)(cx \otimes x + y \otimes x)\\
&= y \otimes x - x \otimes y\\
&= r.
\end{split}
\]
Finally, by \eqref{mualphaxy}, the commutator bracket $[-,-]$ in $(A_\alpha)_{HLie}$ is determined by
\[
[x,y] = \mualpha(x,y) - \mualpha(y,x) = -x.
\]
\qed
\end{example}


\end{document}